\pdfoutput=1
\documentclass[11pt]{article} 
\def\arXiv{1} 
\usepackage{enumitem}
\usepackage{amssymb}
\usepackage{amsbsy}
\usepackage{amsmath}
\usepackage{amsfonts}
\usepackage{latexsym}
\usepackage{graphicx}
\usepackage{color}
\usepackage{xifthen}
\usepackage{xspace}
\usepackage{mathtools}
\usepackage{bbm}
\usepackage{multirow}
\usepackage[normalem]{ulem}
\usepackage{array}
\usepackage[utf8]{inputenc}
\usepackage[T1]{fontenc}
\usepackage{etoolbox}
\newtoggle{restatements}

\newtoggle{heavyplots}

\usepackage{xfrac}
\newcommand{\notarxiv}[1]{foo}
\newcommand{\arxiv}[1]{ba}
\ifdefined \arXiv
	\renewcommand{\arxiv}[1]{#1}%
	\renewcommand{\notarxiv}[1]{\ignorespaces}%
\else%
	\renewcommand{\arxiv}[1]{\ignorespaces}%
	\renewcommand{\notarxiv}[1]{#1}%
\fi%

\usepackage{titlesec}

\usepackage{amsthm}
\usepackage{thmtools}
\usepackage{thm-restate}
\usepackage[linesnumbered,ruled,vlined]{algorithm2e}

\notarxiv{
\usepackage[utf8]{inputenc} %
\usepackage[T1]{fontenc}    %
\usepackage[hidelinks]{hyperref}       %
\usepackage{url}            %
\usepackage{booktabs}       %
\usepackage{amsfonts}       %
\usepackage{nicefrac}       %
\usepackage{microtype}      %
\usepackage{xcolor}
\usepackage{tabularx}
\usepackage[noend]{algpseudocode}
\usepackage{stackengine}
\usepackage{comment}
}

\arxiv{
	\usepackage[dvipsnames]{xcolor}
	\usepackage[top=1in, right=1in, left=1in, bottom=1in]{geometry}
	\usepackage[numbers,square]{natbib}
	\usepackage{url}
	\usepackage[hidelinks]{hyperref}
	\hypersetup{
		colorlinks=true,
		linkcolor=blue!70!black,
		citecolor=blue!70!black,
		urlcolor=blue!70!black}
}

\theoremstyle{plain}

\newtheorem{theorem}{Theorem}
\newtheorem{lemma}{Lemma}

\newtheorem{proposition}{Proposition}
\newtheorem{corollary}{Corollary}
\newtheorem{definition}{Definition}

\theoremstyle{definition}
\newtheorem{remark}{Remark}

\newtheorem*{example*}{Example}

\arxiv{
}

\makeatletter
\newcommand*{\propenum}[1]{%
	\expandafter\@propenum\csname c@#1\endcsname%
}

\newcommand*{\@propenum}[1]{%
	$\ifcase#1\or1\or1'\or2\or3\or42%
	\else\@ctrerr\fi$%
}
\AddEnumerateCounter{\propenum}{\@propenum}{1}
\makeatother

\usepackage{cleveref}

\newcommand{\mc}[1]{\mathcal{#1}}

\newcommand{\wt}[1]{\widetilde{#1}}  %
\newcommand{\what}[1]{\widehat{#1}}  %
\newcommand{\norm}[1]{\left\|{#1}\right\|} %
\newcommand{\lone}[1]{\norm{#1}_1} %
\newcommand{\ltwo}[1]{\norm{#1}_2} %
\newcommand{\linf}[1]{\norm{#1}_\infty} %
\newcommand{\lfro}[1]{\left\|{#1}\right\|_{\rm F}} %
\newcommand{\norms}[1]{\|{#1}\|} %

\newcommand{\opnorm}[1]{\norm{#1}_{\mathrm{op}}}  %
\newcommand{\lones}[1]{\norms{#1}_1} %
\newcommand{\ltwos}[1]{{\norms{#1}}_2} %
\newcommand{\linfs}[1]{\norms{#1}_\infty} %
\newcommand{\R}{\mathbb{R}} %
\newcommand{\N}{\mathbb{N}} %
\newcommand{\E}{\mathbb{E}} %
\newcommand{\ai}[1][i]{A_{#1:}}
\newcommand{\aj}[1][j]{A_{:#1}}
\newcommand{\nnz}{\textup{nnz}}

\usepackage{xcolor}

\SetCommentSty{mycommfont}
\SetKwInput{KwInput}{Input}                %
\SetKwInput{KwOutput}{Output}              %
\SetKwInput{KwReturn}{Return}              %

\newcommand{\InnerLoop}{\mathtt{InnerLoop}}
\newcommand{\RestartedInnerLoop}{\mathtt{RestartedInnerLoop}}

\newcommand{\OuterLoop}{\mathtt{OuterLoop}}
\newcommand{\OuterLoopGen}{\mathtt{OuterLoopStronglyMonotone}}
\newcommand{\FullLoopsm}{\mathtt{FullLoopStronglyMonotone}}

\makeatletter
\let\oldnl\nl%
\newcommand{\nonl}{\renewcommand{\nl}{\let\nl\oldnl}}%
\makeatother

\DeclareMathOperator*{\argmin}{arg\,min}

\providecommand{\sign}{\mathop{\rm sign}}

\newcommand{\ceil}[1]{\left\lceil{#1}\right\rceil}

\newcommand{\half}{\frac{1}{2}}
\newcommand{\defeq}{\coloneqq}

\newcommand{\grad}{\nabla}

\newcommand{\xset}{\mathcal{X}}
\newcommand{\yset}{\mathcal{Y}}
\newcommand{\zset}{\mathcal{Z}}

\newcommand{\eps}{\epsilon}

\renewcommand{\O}[1]{O\left( #1 \right)}
\newcommand{\Otil}[1]{\widetilde{O}( #1 )}
\newcommand{\Otilb}[1]{\widetilde{O}\left( #1 \right)}

\newcommand{\innermid}{\nonscript\;\delimsize\vert\nonscript\;}
\newcommand{\activatebar}{%
	\begingroup\lccode`\~=`\|
	\lowercase{\endgroup\let~}\innermid 
	\mathcode`|=\string"8000
}

\newcommand{\inner}[2]{\left<#1,#2\right>}
\newcommand{\inners}[2]{\big<#1,#2\big>}

\NewDocumentCommand{\prox}{ m m 
O{\alpha}}{\mathrm{Prox}_{#1}^{#3}(#2)}

\newcommand{\gapu}{\textrm{gap}}

\newcommand{\nA}{\norm{A}_{2\rightarrow\infty} }
\newcommand{\ball}{\mathbb{B}}

\newcommand{\trunc}[1][\tau]{\mathsf{T}_{#1}}

\newcommand{\x}{^\mathsf{x}}
\newcommand{\y}{^\mathsf{y}}

\newcommand{\Ex}[1][\xi]{\E}

\usepackage[textsize=scriptsize,textwidth=2cm]{todonotes}

\notarxiv{
	\usepackage[subtle, all=normal, mathspacing=normal, floats=tight, 
	sections=tight]{savetrees}
}

\title{Variance Reduction for Matrix Games}

\arxiv{
\author{Yair Carmon ~~~ Yujia Jin ~~~  Aaron Sidford ~~~ Kevin 
Tian\\
	\texttt{\{\href{mailto:yairc@stanford.edu}{yairc},%
		\href{mailto:yujiajin@stanford.edu}{yujiajin},%
		\href{mailto:sidford@stanford.edu}{sidford},%
		\href{mailto:kjtian@stanford.edu}{kjtian}\}@stanford.edu}}
\date{}
}
\notarxiv{
\author{%
	Yair Carmon, Yujia Jin, Aaron Sidford and Kevin Tian\\
	Stanford University\\
	\texttt{\{yairc,yujiajin,sidford,kjtian\}@stanford.edu}\\
}
}

\begin{document}

\maketitle

\begin{abstract}%
We present a randomized primal-dual algorithm that solves the problem 
$\min_{x} \max_{y} y^\top A x$ to additive error $\epsilon$ in time 
$\nnz(A) + \sqrt{\nnz(A)n}/\epsilon$, for matrix $A$ with larger dimension 
$n$ and $\nnz(A)$ nonzero entries. This improves the best known exact 
gradient methods by a factor of $\sqrt{\nnz(A)/n}$ and is faster than fully 
stochastic gradient methods in the accurate and/or sparse regime 
$\epsilon \le \sqrt{n/\nnz(A)}$. Our results hold for $x,y$ in the simplex 
(matrix games, linear programming) and for $x$ in an $\ell_2$ ball and $y$ 
in the simplex (perceptron / SVM, minimum enclosing ball). Our algorithm 
combines Nemirovski's ``conceptual prox-method'' and a novel 
reduced-variance gradient estimator based on ``sampling from the 
difference'' between the current iterate and a reference point.
\end{abstract}

\section{Introduction}
\label{sec:intro}

Minimax problems---or games---of the form $\min_x \max_y f(x,y)$ are 
ubiquitous in economics, statistics, optimization and machine learning. In 
recent years, minimax formulations for neural network training rose to 
prominence~\citep{GoodfellowPMXWOCB14,MadryMSTV18}, leading to 
intense interest in algorithms for solving large scale minimax games 
\citep{DroriST15,GidelJL16,KolossoskiM17,DaskalakisISZ18,JiinNJ19, 
MertikopoulosZLFC19}. However, the algorithmic toolbox for minimax 
optimization is not as complete as the one for minimization. Variance 
reduction, a technique for improving stochastic gradient estimators by 
introducing control variates, stands as a case in point. A multitude of 
variance reduction schemes exist for finite-sum minimization 
\citep[cf.][]{Johnson13, SchmidtRB17, Zhu17, BottouCN18, FangLLZ18}, and 
their impact on complexity is well-understood~\citep{WoodworthS16}. In 
contrast, only a few works apply variance reduction to finite-sum minimax 
problems~\citep{BalamuruganB16,ShiZY17, 
ChavdarovaGFL19,MishchenkoKSRM19}, and the potential gains from 
variance reduction are not well-understood.

We take a step towards closing this gap by designing variance-reduced 
minimax game solvers that offer strict runtime improvements over 
non-stochastic gradient methods, similar to that of optimal variance 
reduction methods for finite-sum minimization. To achieve this, we focus 
on the fundamental class of bilinear minimax games, 
\begin{equation*}
\min_{x \in \xset} \max_{y \in \yset} {y^\top A x}, ~~\mbox{where}~~A \in 
\R^{m \times n}.
\end{equation*}
In particular, we study the complexity of finding an $\epsilon$-approximate saddle point (Nash equilibrium), namely $x, y$ with
\begin{equation*}
\max_{y' \in \yset}{ (y')^\top A x} - \min_{x' \in \xset} {y^\top 
A x' } \le \epsilon.
\end{equation*}
In the setting where $\xset $ and $\yset$ are both probability simplices, 
the problem corresponds to finding an approximate (mixed) equilbrium in a 
matrix game, a central object in game theory and economics. Matrix games 
are also fundamental to algorithm design due in part to their equivalence to 
linear programming \citep{Dantzig53}. Alternatively, when $\xset$ is an 
$\ell_2$ ball and $\yset$ is a simplex,
solving the corresponding problem finds a maximum-margin linear 
classifier (hard-margin SVM), a fundamental task in machine learning and 
statistics \citep{MinskyP87}.
 We refer to the former as an $\ell_1$-$\ell_1$ game and the latter as an 
 $\ell_2$-$\ell_1$ game; our primary focus is to give improved algorithms 
 for these domains.

\subsection{Our Approach}

Our starting point is Nemirovski's ``conceptual 
prox-method''~\citep{Nemirovski04} for solving $\min_{x\in\xset} 
\max_{y\in\yset} f(x,y)$, 
where $f:\xset\times\yset\to \R$ is convex in $x$ and concave in $y$. The 
method solves a 
sequence of subproblems 
parameterized by $\alpha>0$, each of the form 
\begin{equation}\label{eq:relaxed-proxproblem}
\text{find $x,y$ s.t.\ $\forall x',y'$ }
\inner{\grad_x f(x,y)}{x-x'} - \inner{\grad_y f(x,y)}{y-y'} \le 
\alpha V_{x_0}(x') + \alpha V_{y_0}(y')
\end{equation}
for some $(x_0,y_0)\in\xset\times \yset$, where $V_a(b)$ is a 
norm-suitable Bregman divergence from $a$ to $b$: squared Euclidean 
distance for $\ell_2$ and KL divergence for $\ell_1$. Combining each 
subproblem solution with an extragradient step, the prox-method solves 
the 
original problem to $\epsilon$ accuracy by solving 
$\Otil{\alpha/\epsilon}$ subproblems.\footnote{ 
	More precisely, the required number of subproblem solutions is at most 
	$\Theta 
	\cdot \frac{\alpha}{\epsilon}$, where $\Theta$ is a ``domain size'' 
	parameter that 
	depends on $\xset$, $\yset$, and the Bregman divergence $V$ 
	(see~\Cref{sec:preliminaries}). In the 
	 $\ell_1$ and $\ell_2$ settings considered in this paper, we have the 
	 bound 
	$\Theta \le \log(nm)$ and we use the $\wt{O}$ 
	notation to suppress terms logarithmic in $n$ and $m$.
	However, in other settings---e.g., $\ell_\infty$-$\ell_1$ 
	games~\cite[cf.][]{Sherman17,SidfordTi18}---making the  parameter 
	$\Theta$ scale logarithmically with the problem dimension is far 
	more difficult.
} 
(Solving~\eqref{eq:relaxed-proxproblem} with $\alpha=0$ is equivalent to 
to solving $\min_{x\in\xset} \max_{y\in\yset} f(x,y)$.)

Our first contribution is showing that if a stochastic unbiased gradient 
estimator $\tilde{g}$ satisfies the ``variance'' bound
\begin{equation}\label{eq:var-bound}
\E \norm{\tilde{g}(x,y) - \grad f(x_0,y_0)}_*^2 \le L^2 
\norm{x-x_0}^2 
+ L^2 \norm{y-y_0}^2
\end{equation}
for some $L>0$, then $O(L^2 / \alpha^2)$ regularized stochastic 
mirror descent steps using $\tilde{g}$ 
solve~\eqref{eq:relaxed-proxproblem} in a suitable probabilistic sense. We 
call unbiased gradient estimators that satisfy~\eqref{eq:var-bound} 
``centered.''

Our second contribution is the construction of ``centered'' gradient 
estimators for $\ell_1$-$\ell_1$ and 
$\ell_2$-$\ell_1$ bilinear games, where $f(x,y)=y^\top A x$.
Our $\ell_1$ estimator has the following form. Suppose we wish to 
estimate $g\x = A^\top y$ (the gradient of $f$ w.r.t.\ $x$), and we already 
have $g_0\x = A^\top y_0$. Let $p\in\Delta^m$ be some distribution over 
$\{1,\ldots,m\}$, draw $i\sim p$ and set 
\[
\tilde{g}\x = g_0\x + A_{i:}\frac{[y]_i - [y_0]_i}{p_i},
\] 
where $\ai$ is the $i$th column of $A^\top$. This form 
is familiar from variance reduction 
techniques~\citep{Johnson13,XiaoZ14,Zhu17}, that typically use a fixed 
distribution 
$p$. In our setting, however, a fixed $p$ will not produce sufficiently low 
variance. Departing from prior variance-reduction work and building 
on~\cite{,GrigoriadisK95,ClarksonHW10}, we choose $p$ \emph{based on 
$y$} according 
to 
\[
p_i(y) = \frac{\big|[y]_i - [y_0]_i\big|}{\lone{y-y_0}},
\]
yielding exactly the variance bound we require. We call this technique 
``sampling from the difference.''

For our $\ell_2$ gradient estimator, we sample from the \emph{squared} 
difference, drawing $\xset$-block coordinate $j\sim q$, where
\arxiv{
\[
q_j(x) = \frac{([x]_j - [x_0]_j)^2}{\ltwo{x-x_0}^2}.
\]
}
\notarxiv{
$q_j(x) = {([x]_j - [x_0]_j)^2}/{\ltwo{x-x_0}^2}$.
}
To strengthen our results for $\ell_2$-$\ell_1$ games, we consider a 
refined version of the ``centered'' criterion~\eqref{eq:var-bound} which 
allows regret analysis using local 
norms~\cite{Shalev-Shwartz12,ClarksonHW10}. To further facilitate this 
analysis we follow~\cite{ClarksonHW10} and introduce gradient clipping. 
We extend our proofs to show that stochastic regularized mirror descent 
can solve~\eqref{eq:relaxed-proxproblem} despite the (distance-bounded) 
bias caused by gradient clipping.

Our gradient estimators attain the 
bound~\eqref{eq:var-bound} with $L$ equal to the Lipschitz constant of 
$\grad f$. Specifically, 
\begin{equation}\label{eq:L-vals}
L= \begin{cases}
\max_{ij}|A_{ij}| & \text{in the $\ell_1$-$\ell_1$ setup}\\
\max_{i} \ltwo{\ai} & \text{in the $\ell_2$-$\ell_1$ setup.}\\
\end{cases}
\end{equation}

\subsection{Method complexity compared with prior art}

As per the discussion above, to achieve accuracy $\epsilon$ our algorithm 
solves $\Otil{\alpha/\epsilon}$ subproblems. Each subproblem takes 
$O(\nnz(A))$ time for computing two exact gradients (one for variance 
reduction and one for an extragradient step), plus an additional $(m + 
n)L^2/\alpha^2$ time for the inner mirror descent iterations, with $L$ as 
in~\eqref{eq:L-vals}. The total 
runtime is therefore
\begin{equation*}
\Otilb{\bigg(\nnz(A) + \frac{(m + n)L^2}{\alpha^2} 
\bigg)\frac{\alpha}{\epsilon}}.
\end{equation*}
By setting $\alpha$ optimally to be $\max\{\epsilon, L\sqrt{(m + n)/ 
\nnz(A)}\}$, we obtain the runtime
\begin{equation}\label{eq:our-time}
\Otil{\nnz(A) + \sqrt{\nnz(A)\cdot(m + n)} \cdot L \cdot \epsilon^{-1}}.
\end{equation}

\paragraph{Comparison with mirror-prox and dual extrapolation.}
\citet{Nemirovski04} instantiates his conceptual prox-method by solving 
the relaxed proximal problem~\eqref{eq:relaxed-proxproblem} with 
$\alpha=L$ in time $O(\nnz(A))$, where $L$ is the Lipschitz constant of 
$\grad f$, as given in~\eqref{eq:L-vals}. The 
total complexity of the resulting method is therefore
\begin{equation}\label{eq:extragrad-time}
\Otil{\nnz(A) \cdot L \cdot \eps^{-1}}.
\end{equation}
The closely related dual extrapolation method of~\citet{Nesterov07} attains 
the same rate of convergence. We refer to the running 
time~\eqref{eq:extragrad-time} as \emph{linear} since it scales linearly with 
the problem description size $\nnz(A)$. Our running time 
guarantee~\eqref{eq:our-time} is never worse 
than~\eqref{eq:extragrad-time} by more than a constant factor, and 
improves 
on~\eqref{eq:extragrad-time}  
when $\nnz(A)=\omega(n+m)$, i.e.\ whenever $A$ is not extremely sparse. 
In that regime, our method uses $\alpha \ll L$, hence solving a harder 
version of~\eqref{eq:relaxed-proxproblem} than possible for mirror-prox.

\paragraph{Comparison with sublinear-time methods}
Using a randomized algorithm, \citet{GrigoriadisK95} solve 
$\ell_1$-$\ell_1$ bilinear games in time
\begin{equation}\label{eq:sublinear-time}
\Otil{(m + n) \cdot L^2 \cdot \eps^{-2}},
\end{equation}
and~\citet{ClarksonHW10} extend this result to $\ell_2$-$\ell_1$ bilinear 
games, with the values of $L$ as in~\eqref{eq:L-vals}. Since these runtimes 
scale with $n+m\le \nnz(A)$, we refer to them as \emph{sublinear}. Our
guarantee improves on the guarantee~\eqref{eq:sublinear-time} when 
$(m + n) \cdot L^2 \cdot \eps^{-2} \gg \nnz(A)$, i.e.\ 
whenever~\eqref{eq:sublinear-time} is not truly sublinear. %

\arxiv{\paragraph*{}} 
Our method carefully balances linear-time extragradient steps with cheap 
sublinear-time stochastic gradient steps. Consequently, our runtime 
guarantee~\eqref{eq:our-time} inherits strengths from both the linear and 
sublinear runtimes. First, our runtime scales linearly with $L/\epsilon$ 
rather than quadratically, as does the linear 
runtime~\eqref{eq:extragrad-time}. Second, while our runtime is not strictly
sublinear, its component proportional to $L/\epsilon$ is  
$\sqrt{\nnz(A)(n+m)}$, which is sublinear in $\nnz(A)$. 

Overall, our method 
offers the 
best runtime guarantee in the literature in the regime
\begin{equation*}
\frac{\sqrt{\nnz(A)(n+m)}}{\min\{n,m\}^{\omega}} \ll \frac{\epsilon}{L} \ll 
\sqrt{\frac{n+m}{\nnz(A)}},
\end{equation*}
where the lower bound on $\epsilon$ is due to the best known theoretical 
runtimes of interior point methods: $\Otil{\max\{n,m\}^\omega 
\log(L/\epsilon)}$~\cite{CohenLS18} and $\Otil{\nnz(A) + \min\{n,m\}^2) 
\sqrt{\min\{n,m\}} \log(L/\epsilon)}$~\cite{LeeS15}, where $\omega$ is the 
(current) matrix multiplication exponent.

In the square dense case (i.e.\ $\nnz(A)\approx n^2 =m^2$), we improve 
on the  accelerated runtime~\eqref{eq:extragrad-time} by a factor of 
$\sqrt{n}$, the same improvement that optimal variance-reduced 
finite-sum minimization methods achieve over the fast gradient 
method~\cite{XiaoZ14,Zhu17}. %

\arxiv{
\subsection{Additional contributions}\label{ssec:intro-additional}
We extend our development in five ways. First, we show how to combine 
restarting with variance reduction in order to compute exact proximal 
points to high accuracy. This technique applies to any function $f$ with a 
centered gradient estimator (rather than the bilinear functions considered 
so far). Second, we describe an extension of our results to  
``composite'' saddle point problems of the form 
$\min_{x\in\xset}\max_{y\in\yset} \left\{f(x,y) + \phi(x) - \psi(y)\right\}$, 
where $f$ admits a centered gradient estimator and $\phi,\psi$ are 
``simple'' convex functions.
Third, we adapt our framework to achieve linear convergence when $\phi$ 
and $\psi$ are relatively strongly convex in the sense of~\cite{LuFN18}. 
Fourth, we show that for finite-sum saddle point problems the ``SVRG'' 
gradient estimator~\cite{Johnson13} is centered, and derive corresponding 
runtimes guarantees.
 Fifth, we describe a number of alternative 
centered gradient estimators for the $\ell_2$-$\ell_1$ and 
$\ell_2$-$\ell_2$ setups.
}

\subsection{Related work}\label{ssec:related}

Matrix games, the canonical form of discrete zero-sum games, have long 
been studied in economics \cite{VonNeumann28}. The 
classical mirror descent (i.e.\ no-regret) method yields an algorithm with 
running time $\Otil{\nnz(A)L^2\eps^{-2}}$~\citep{NemirovskyY83}. 
Subsequent work~\cite{GrigoriadisK95, Nemirovski04, Nesterov07, 
ClarksonHW10} improve this runtime as described above. Our work builds 
on the extragradient scheme of \citet{Nemirovski04} as well as the gradient 
estimation and clipping technique of~\citet{ClarksonHW10}.

\citet{BalamuruganB16} apply standard variance reduction 
\citep{Johnson13} to bilinear $\ell_2$-$\ell_2$ games by sampling 
elements proportional to squared matrix entries. Using 
proximal-point acceleration they obtain a runtime of 
$\Otil{\nnz(A)+\|A\|_\mathrm{F} 
\sqrt{\nnz(A)\max\{m,n\}}\eps^{-1}\log\frac{1}{\eps}}$, a rate we recover 
 using our algorithm (\Cref{ssec:l2l2}). However, in this 
setting the mirror-prox method has runtime 
$\Otil{\opnorm{A}\nnz(A)\eps^{-1}}$, which may be better than the result 
of 
\cite{BalamuruganB16} by a factor of $\sqrt{mn/\nnz(A)}$ due to the  
discrepancy in the norm of $A$.
Naive application of \cite{BalamuruganB16} to $\ell_1$ domains results in 
even greater potential losses. \citet{ShiZY17} extend the method 
of~\cite{BalamuruganB16} to smooth functions using general Bregman 
divergences, but their extension is unaccelerated and appears limited to a 
$\eps^{-2}$ rate.

\citet{ChavdarovaGFL19} propose a variance-reduced extragradient 
 method with applications to generative adversarial training. In contrast to 
 our algorithm, which performs extragadient steps in the outer loop, the 
 method of \cite{ChavdarovaGFL19} performs stochastic extragradient steps 
 in the inner loop, using finite-sum variance reduction as 
 in~\cite{Johnson13}. \citet{ChavdarovaGFL19} analyze their method in 
 the  convex-concave setting, showing improved stability over direct 
 application of the extragradient method to noisy gradients. However, their 
 complexity guarantees are worse than those of linear-time methods. 
 Following up on~\citep{ChavdarovaGFL19}, \citet{MishchenkoKSRM19} 
 propose to reduce the variance of the stochastic extragradient method by 
 using the same stochastic sample for both the gradient and extragradient 
 steps. In 
 the Euclidean strongly convex case, they show a convergence guarantee 
 with a relaxed variance assumption, and in the noiseless full-rank bilinear 
 case they recover the guarantees of~\citep{MokhtariOP19}. In the general 
 convex
 case, however, they only show an $\epsilon^{-2}$ rate of convergence.
 
 \arxiv{
 \subsection{Paper outline} We define our notation in 
 Section~\ref{sec:preliminaries}. In Section~\ref{ssec:outerloop}, we review 
 Nemirovski's conceptual prox-method and introduce the notion of a 
 relaxed proximal oracle; we implement such oracle using variance-reduced 
 gradient estimators in Section~\ref{ssec:alphaprox}. In 
 Section~\ref{sec:apps}, we construct these gradient estimators for the 
 $\ell_1$-$\ell_1$, $\ell_2$-$\ell_1$ as well as $\ell_2$-$\ell_2$ domain 
 settings, and complete the analyses of the corresponding algorithms. 
 Finally, in Section~\ref{sec:ext} we give our additional contributions 
 described in Section~\ref{ssec:intro-additional} above.
}

\notarxiv{
	 \subsection{Paper outline and additional contributions} 
	 We define our notation in 
	Section~\ref{sec:preliminaries}. In Section~\ref{ssec:outerloop}, we 
	review 
	Nemirovski's conceptual prox-method and introduce the notion of a 
	relaxed proximal oracle; we implement such oracle using 
	variance-reduced 
	gradient estimators in Section~\ref{ssec:alphaprox}. In 
	Section~\ref{sec:apps}, we construct these gradient estimators for the 
	$\ell_1$-$\ell_1$ and $\ell_2$-$\ell_1$ domain 
	settings, and complete the analyses of the corresponding algorithms; 
	in \Cref{ssec:l2l2} we provide analogous treatment for the 
	$\ell_2$-$\ell_2$  setting, recovering the results 
	of~\citep{BalamuruganB16}.
	
	In~\Cref{sec:ext} we provide three additional contributions:  
	variance-reduction-based computation of proximal points for arbitrary 
	convex-concave functions (\Cref{ssec:highprecision}); extension of our 
	results to  
	``composite'' saddle point problems of the form 
	$\min_{x\in\xset}\max_{y\in\yset} \left\{f(x,y) + \phi(x) - 
	\psi(y)\right\}$, 
	where $f$ admits a centered gradient estimator and $\phi,\psi$ are 
	``simple'' convex functions (\Cref{ssec:composite}); and a number of 
	alternative centered gradient estimators for the $\ell_2$-$\ell_1$ and 
	$\ell_2$-$\ell_2$ settings (\Cref{sec:additional}).
} %
\section{Notation}
\label{sec:preliminaries}

\paragraph{Problem setup.} 
A \emph{setup} is the triplet $(\zset, \norm{\cdot}, r)$ where: (i) $\zset$ is 
a compact and convex subset of $\R^n \times \R^m$, (ii) $\norm{\cdot}$ 
is a norm on $\zset$ and (iii) $r$ is 1-strongly-convex w.r.t.\ $\zset$ and 
$\norm{\cdot}$, i.e.\ such that 
$
r(z') \ge r(z) + \inner{\grad r(z)}{z-z'} + \half \norm{z'-z}^2
$
for all $z,z'\in\zset$.\footnote{ For non-differentiable $r$, \arxiv{we 
	define }\notarxiv{let }$\inner{\grad r(z)}{w}\defeq\sup_{\gamma \in 
	\partial r(z)} 
	\inner{\gamma}{w}$, where $\partial r(z)$ is the subdifferential of $r$ at 
	$z$.}
We call $r$ the \emph{distance generating function} 
and denote the \emph{Bregman divergence} associated with it by
\begin{equation*}
V_{z}(z') \defeq r(z') - r(z) - \inner{\grad r(z)}{z'-z} \ge \half \norm{z'-z}^2.
\end{equation*}
We also 
denote $\Theta \defeq \max_{z'} r(z') - \min_{z} r(z)$ and assume it is finite.

\paragraph{Norms and dual norms.}
We write $\mathcal{S}^*$ for the set of linear functions on $\mathcal{S}$. 
For $\zeta \in \zset^*$ we define the dual norm of $\norm{\cdot}$ as  
$\norm{\zeta}_* \defeq \max_{\norm{z}\le 1} \inner{\zeta}{z}$. For 
$p\ge 1$ we write the $\ell_p$ norm $\norm{z}_p = (\sum_i z_i^p)^{1/p}$ 
with $\linf{z} = \max_i |z_i|$. The dual norm of $\ell_p$ is $\ell_q$ 
with $q^{-1} = 1-p^{-1}$. 

\paragraph{Domain components.}
We assume $\zset$ is of the form $\xset\times\yset$ for convex and 
compact sets $\xset\subset \R^n$ and $\yset\subset\R^m$. Particular 
sets of interest are the simplex $\Delta^d = \{v\in\R^d \mid \lone{v}=1, 
v \ge 0 \}$
 and the Euclidean ball $\ball^d = \{v\in\R^d \mid \ltwo{v} \le 1\}$.
For any vector in $z\in \R^n\times \R^m$, 
\begin{equation*}
\text{we write $z\x$ and $z\y$ for 
the first $n$ and last $m$ coordinates of $z$, respectively.} 
\end{equation*}
When totally clear from context, we sometimes refer to the $\xset$ and 
$\yset$ components of $z$ directly as $x$ and $y$. We write the $i$th 
coordinate of vector $v$ as $[v]_i$.

\paragraph{Matrices.}
We consider a matrix $A\in\R^{m\times n}$ and write $\nnz(A)$ for the 
number of its nonzero entries. For $i\in[n]$ and $j\in[m]$ we write 
$\ai$, $\aj$ and $A_{ij}$ for the corresponding row, column and entry, 
respectively.\footnote{ For $k\in\N$, we let $[k]\defeq \{1,\ldots,k\}$.} We 
consider the matrix norms $\norm{A}_{\max} \defeq  
\max_{ij}|A_{ij}|$, $\norm{A}_{p \rightarrow q}\defeq \max_{\norm{x}_p \le  
1} \norm{Ax}_q$ and $\lfro{A} \defeq (\sum_{i, j} A_{ij}^2)^{1/2}$.

\arxiv{\newpage}

\section{Primal-dual variance reduction framework}
\label{sec:framework}

\newcommand{\gap}{\mathrm{Gap}}
\newcommand{\veps}{\varepsilon}

In this section, we establish a framework for solving the saddle point 
problem 
\begin{equation}\label{eq:problem-gen}
\min_{x\in\xset}\max_{y\in\yset} f(x,y),
\end{equation}
where $f$ is convex in $x$ and concave $y$, and admits a 
(variance-reduced) stochastic estimator for the continuous and 
monotone\footnote{ A mapping $q:\zset\to\zset^*$ is monotone if and 
only if $\inner{q(z')-q(z)}{z'-z}\ge0$ for all $z,z'\in\zset$; $g$ is  
monotone due to convexity-concavity of $f$.} gradient mapping
\[
g(z) = g(x,y) \defeq \left(\nabla_{x} f(x,y),-\nabla_{y} f(x,y)\right).
\]
Our goal is to find an $\epsilon$-approximate saddle point (Nash 
equilibrium), i.e.\ 
$z\in\zset\defeq \xset\times \yset$ such that
\begin{equation}\label{eq:gap-def}
\gap(z) \defeq \max_{y' \in \yset} f(z\x,y') - \min_{x' \in \xset} f(x', z\y) 
\le \epsilon.
\end{equation}
We achieve this by generating a sequence $z_1,z_2,\ldots,z_k$ such that 
$\frac{1}{K}\sum_{k=1}^K \inner{g(z_k)}{z_k-u}\le \epsilon$ for every 
$u\in\zset$ and using the fact that
\begin{equation}\label{eq:gap-bound}
\gap\left(\frac{1}{K}\sum_{k=1}^K z_k\right)\le \max_{u\in \zset} 
\frac{1}{K}\sum_{k=1}^K \inner{g(z_k)}{z_k-u}
\end{equation}
due to convexity-concavity of $f$ (see proof in 
Appendix~\ref{app:gap-bound}).

In Section~\ref{ssec:outerloop} we define the notion of a (randomized) 
\emph{relaxed proximal oracle}, and describe how Nemirovski's 
mirror-prox method leverages it to solve the 
problem~\eqref{eq:problem-gen}. In Section~\ref{ssec:alphaprox} we define 
a class of \emph{centered} gradient estimators, whose variance is 
proportional to the squared distance from a reference point. Given such a 
centered gradient estimator, we show that a regularized stochastic mirror 
descent 
scheme constitutes a relaxed proximal oracle. For a technical reason, we  
limit our oracle guarantee in Section~\ref{ssec:alphaprox} to the bilinear 
case $f(x,y)=y^\top A x$, which suffices for the applications in 
Section~\ref{sec:apps}. We lift this limitation in~\Cref{ssec:highprecision}, 
where we show a different oracle 
implementation that is valid  for general convex-concave $f$, with only a 
logarithmic increase in complexity.

\subsection{The mirror-prox method with a randomized oracle}
\label{ssec:outerloop}

Recall that we assume the space $\zset=\xset\times\yset$ is equipped 
with a norm $\norm{\cdot}$ and distance generating function 
$r:\zset\to\R$ that is $1$-strongly convex w.r.t.\ $\norm{\cdot}$ and has 
range $\Theta$. We write the induced Bregman divergence as 
$V_{z}(z')=r(z')-r(z)-\inner{\grad r(z)}{z'-z}$. We use the following fact 
throughout the paper: by definition, the Bregman 
divergence satisfies, for any $z, z', u\in\zset$,
\begin{equation}\label{eq:three-point}
-\inner{\nabla V_z(z')}{z' - u} = V_z(u) - V_{z'}(u) - V_z(z').
\end{equation}
For any $\alpha > 0$ we define 
the $\alpha$-proximal mapping 
$\prox{z}{g}$ to be the solution of the variational 
inequality corresponding to the strongly monotone operator $g+\alpha 
\grad V_{z}$, 
i.e.\ the unique $z_\alpha\in\zset$ such that 
$\inner{g(z_\alpha)+\alpha\nabla V_z(z_\alpha)}{z_\alpha-u} \le 0$ for all 
$u\in\zset$~\citep[cf.][]{Eckstein93}. Equivalently 
(by~\eqref{eq:three-point}), 
\begin{equation}\label{eq:prox-def}
\prox{z}{g}  \defeq \text{the unique $z_\alpha\in\zset$ s.t.}~
\inner{g(z_\alpha)}{z_\alpha-u} \le \alpha V_z(u)-\alpha 
V_{z_\alpha}(u)-\alpha 
V_z(z_\alpha)~~\forall 
u\in\zset.
\end{equation}
When 
$V_{z}(z')=V_{x}\x(x')+V_{y}\y(y')$, $\prox{z}{g}$ is also the unique 
solution of the saddle point problem 
\begin{equation*}
\min_{x'\in\xset}\max_{y'\in\yset}\left\{f(x',y')+\alpha V_{x}\x(x') - 
\alpha 
V_{y}\y(y')\right\}. 
\end{equation*}

Consider iterations of the form $z_{k} = \prox{z_{k-1}}{g}$, with 
$z_0=\argmin_z r(z)$.  Averaging the definition~\eqref{eq:prox-def} over 
$k$, using the bound~\eqref{eq:gap-bound} and the nonnegativity of 
Bregman divergences gives
\begin{equation*}
\gap\left(\frac{1}{K}\sum_{k=1}^K z_k\right)\le 
\max_{u\in \zset} \frac{1}{K}\sum_{k=1}^K \inner{g(z_k)}{z_k-u}
\le \max_{u\in \zset} \frac{\alpha 
\left(V_{z_0}(u)-V_{z_K}(u)\right)}{K} \le 
\frac{\alpha\Theta}{K}.
\end{equation*}
Thus, we can find an $\epsilon$-suboptimal point in 
$K=\alpha\Theta/\epsilon$ exact proximal steps. However, computing 
$\prox{z}{g}$ exactly may be as difficult as solving the original 
problem. 
\citet{Nemirovski04} proposes a relaxation of the exact proximal mapping, 
which we slightly extend to include the possibility of randomization, and 
formalize in the following.

\begin{definition}[($\alpha,\veps$)-relaxed proximal oracle]
\label{def:alphaprox}
Let $g$ be a monotone operator 
and $\alpha,\veps>0$. 
An \emph{($\alpha,\veps$)-relaxed proximal oracle} for $g$ is a 
(possibly randomized) mapping $\mc{O}:\zset\to\zset$ such that 
$z'=\mc{O}(z)$ satisfies
\begin{equation*}
\E\left[ \max_{u \in \zset}\big\{\inner{g(z')}{z' - u} - \alpha 
V_z(u)\big\}\right] \leq \veps.
\end{equation*}
\end{definition}

Note that $\mc{O}(z)=\prox{z}{g}$ is an $(\alpha,0)$-relaxed proximal 
oracle. Algorithm~\ref{alg:outerloop} describes the ``conceptual 
prox-method'' of~\citet{Nemirovski04}, which recovers the error guarantee 
of exact proximal iterations. The $k$th iteration consists of (i) a relaxed 
proximal oracle call producing $z_{k-1/2} = \mc{O}(z_{k-1})$, and (ii) a 
\emph{linearized} proximal (mirror) step where we replace $z\mapsto g(z)$ 
with the constant function $z\mapsto g(z_{k-1/2})$, producing $z_{k} =
\prox{z_{k-1}}{g(z_{k-1/2})}$. 
\arxiv{%
 We now state and prove the convergence guarantee for the prox-method, 
first shown in~\citep{Nemirovski04}.
}%
\notarxiv{
	We now state the convergence guarantee for the mirror-prox 
	method, first shown in~\citep{Nemirovski04} (see~\Cref{app:outerloop} 
	for a simple proof).
}

\SetKwInOut{Parameter}{Parameters}
\SetKwComment{Comment}{$\triangleright$\ }{}
\SetCommentSty{color{black}}

\begin{algorithm}
	\label{alg:outerloop}
	\DontPrintSemicolon
	\KwInput{$(\alpha,\veps)$-relaxed proximal oracle $\mathcal{O}(z)$ for 
	gradient mapping $g$, distance-generating $r$}
	\Parameter{Number of iterations $K$}
	\KwOutput{Point $\bar{z}_K$ with $\E\,\gap(\bar{z}) \le 
	\frac{\alpha\Theta}{K} + 
	\veps$}
	$z_0 \gets \argmin_{z\in\zset} r(z)$ \;
	\For{$k = 1, \ldots, K$}
	{
		$z_{k-1/2} \leftarrow \mathcal{O}(z_{k - 1})$ \Comment*[f]{We 
		implement $\mathcal{O}(z_{k-1})$ by calling $\InnerLoop(z_{k-1}, 
		\tilde{g}_{z_{k - 1}}, \alpha)$} \; 
		$z_k \leftarrow \prox{z_{k-1}}{g(z_{k-1/2})}
		= \argmin_{z \in \zset}\left\{ 
		\inner{g\left(z_{k-1/2}\right)}{z} + 
		\alpha V_{z_{k-1}}(z)\right\}$\; \label{line:outer-extragrad}
	}
	\Return $\bar{z}_K = \frac{1}{K}\sum_{k=1}^K z_{k-1/2}$
	\caption{$\OuterLoop(\mathcal{O})$ (\citet{Nemirovski04})}
\end{algorithm}

\begin{restatable}[Prox-method convergence via 
oracles]{proposition}{restateOuterLoop}
\label{prop:outerloopproof}
Let $\mathcal{O}$ be an ($\alpha$,$\veps$)-relaxed proximal oracle with 
respect to gradient mapping $g$ and distance-generating function $r$ with 
range at most $\Theta$. %
Let $z_{1/2}, z_{3/2}, \ldots, z_{K-1/2}$ be the iterates of 
Algorithm~\ref{alg:outerloop} and let $\bar{z}_K$ be its output. Then
\begin{equation*}
\E\,\gap(\bar{z}_K)\le 
\E \max_{u\in \zset} \frac{1}{K}\sum_{k=1}^K 
\inner{g(z_{k-1/2})}{z_{k-1/2}-u}
\le 
\frac{\alpha\Theta}{K} + \veps.
\end{equation*}
\end{restatable}
\newcommand{\outloopproof}{
\begin{proof}
Fix iteration $k$, and note that by the definition~\eqref{eq:prox-def}, 
$z_k=\prox{z_{k-1}}{g(z_{k-1/2})}$ satisfies
\begin{equation*}
\inner{g(z_{k-1/2})}{z_{k} - u} \leq \alpha\left(V_{z_{k - 1}}(u) - V_{z_k}(u) 
- 
V_{z_{k - 1}}(z_k)\right)~~\forall u\in\zset.
\end{equation*}
Summing over $k$, writing \arxiv{$\inner{g(z_{k-1/2})}{z_{k} - u} = 
\inner{g(z_{k-1/2})}{z_{k-1/2} - u} 
- \inner{g(z_{k-1/2})}{z_{k-1/2} - z_{k}}$}
\notarxiv{
\begin{equation*}
\inner{g(z_{k-1/2})}{z_{k} - u} = 
\inner{g(z_{k-1/2})}{z_{k-1/2} - u} 
- \inner{g(z_{k-1/2})}{z_{k-1/2} - z_{k}}
\end{equation*}
} and rearranging yields
\begin{equation*}
\sum_{k=1}^K\inner{g(z_{k-1/2})}{z_{k-1/2} - u} \leq \alpha V_{z_0}(u)  +
 \sum_{k=1}^K \left[ \inner{g(z_{k-1/2})}{z_{k-1/2} - z_{k}}- \alpha 
V_{z_{k - 1}}(z_k)\right]
\end{equation*}
for all $u\in\zset$. Note that since $z_0$ minimizes $r$, $V_{z_0}(u) = 
r(u)-r(z_0) \le \Theta$ for all $u$. Therefore, maximizing the above display 
over $u$ and afterwards taking expectation gives
\begin{equation*}
\E\max_{u\in\zset}\sum_{k=1}^K\inner{g(z_{k-1/2})}{z_{k-1/2} - u} \leq 
\alpha\Theta  +
\sum_{k=1}^K \E \left[ \inner{g(z_{k-1/2})}{z_{k-1/2} - z_{k}}- \alpha 
V_{z_{k - 1}}(z_k)\right].
\end{equation*}
Finally, by Definition~\ref{def:alphaprox}, $\E \left[ 
\inner{g(z_{k-1/2})}{z_{k-1/2} - z_{k}}- \alpha 
V_{z_{k - 1}}(z_k)\right] \le \veps$ for every $k$, and and the result follows 
by dividing by $K$ and using the bound~\eqref{eq:gap-bound}.
\end{proof}
}
\arxiv{\outloopproof}

\subsection{Implementation of an $(\alpha,0)$-relaxed proximal 
oracle}
\label{ssec:alphaprox}

We now explain how to use stochastic variance-reduced gradient 
estimators to design an efficient ($\alpha,0$)-relaxed proximal oracle.
We begin by introducing the bias and variance properties of the estimators 
we require.

\begin{definition}
	\label{def:w-l-centered}
	Let $z_0\in\zset$ and $L>0$. A stochastic gradient estimator 
	$\tilde{g}_{z_0}:\zset\rightarrow\zset^*$ %
	is called %
	\emph{$(z_0,L)$-centered} for $g$ if for all $z\in\zset$
	\begin{enumerate}
		\item $\Ex{} \left[\tilde{g}_{z_0}(z)\right]  = g(z)$,
		\item $\E \norm{\tilde{g}_{z_0}(z) - g(z_0)}_*^2\le L^2 \norm{z- 
			z_0}^2$.%
	\end{enumerate}
\end{definition}

\begin{example*}
	When the problem~\eqref{eq:problem-gen} has finite-sum structure 
	$f(x,y)=\frac{1}{K}\sum_{k=1}^K f_k(x,y)$ and the gradient mappings 
	$g_k 
	= (\grad_x f_k - 
	\grad_y f_k)$ are $L$-Lipschitz for every $k$, sampling $k$ uniformly 
	from $[K]$ and letting 
	$\tilde{g}_{w_0}(w)=g_k(w)-g_k(w_0)+\frac{1}{K}\sum_{k\in[K]}g_k(w_0)$ 
	gives an $(w_0, L)$-centered gradient estimator; this is the saddle-point 
	version of the well-known reduced 
	variance gradient estimator of~\cite{Johnson13}. See 
	Section~\ref{ssec:finitesum} for more details and resulting 
	runtime guarantees.
\end{example*}

\begin{lemma}\label{lem:centered-var}
	A $(z_0,L)$-centered estimator for $g$ satisfies
	$\E \norm{\tilde{g}_{z_0}(z) - g(z)}_*^2\le (2L)^2 \norm{z- 
		z_0}^2$.
\end{lemma}
\begin{proof}
	Writing $\tilde{\delta} = \tilde{g}_{z_0}(z) - g(z_0)$, we have $\E 
	\tilde{\delta} = g(z)-g(z_0)$ by the first centered estimator property. 
	Therefore,
	\vspace{-6pt}
	\begin{equation*}
	\E \norm{\tilde{g}_{z_0}(z) - g(z)}_*^2
	=
	\E \norms{\tilde{\delta}- \E\tilde{\delta}}_*^2 
	\stackrel{(i)}{\le}
	2\E \norms{\tilde{\delta}}_*^2 +2\norms{\E\tilde{\delta}}_*^2
	\stackrel{(ii)}{\le}
	4 \E \norms{\tilde{\delta}}_*^2
	\stackrel{(iii)}{\le}
	(2L)^2 \norm{z-z_0}^2,
	\end{equation*} 
	where the bounds follow from $(i)$ the triangle inequality, $(ii)$ Jensen's 
	inequality and $(iii)$ the second centered estimator property.
\end{proof}

\begin{remark}\label{rem:lipschitz}
	A gradient mapping that admits a $(z,L)$-centered gradient estimator for 
	every $z\in\zset$ is $2L$-Lipschitz, since by Jensen's inequality and 
	Lemma~\ref{lem:centered-var} we have for all $w\in\zset$
	\[
	\norm{g(w)-g(z)}_* = \norm{\E \tilde{g}_z(w)-g(z)}_*
	\le  (\E \norm{\tilde{g}_z(w)-g(z)}_*^2)^{1/2} \le 2L\norm{w-z}.
	\]
\end{remark}

\begin{remark}
	Definition~\ref{def:w-l-centered} bounds the gradient variance using the 
	distance to the reference point. Similar bounds are used in variance 
	reduction for bilinear saddle-point problems 
	with Euclidean norm~\cite{BalamuruganB16}, as well as 
	for finding stationary 
	points in smooth nonconvex finite-sum 
	problems~\cite{ZhuH16,ReddiHSPS16,FangLLZ18,ZhouXG18}. However, 
	known variance reduction methods for smooth convex finite-sum 
	minimization require stronger bounds \cite[cf.][Section 2.1]{Zhu17}.
\end{remark}

With the variance bounds defined, we describe 
Algorithm~\ref{alg:innerloop} which (for the bilinear case) implements a 
relaxed proximal oracle. The algorithm is stochastic 
mirror descent with an additional regularization term around the 
initial point $w_0$. Note that we do not perform extragradient steps in this 
stochastic method. When combined with a centered gradient estimator, the 
iterates of Algorithm~\ref{alg:innerloop} provide the following guarantee, 
which is one of our key technical contributions.

\begin{algorithm}
	\label{alg:innerloop}
	\DontPrintSemicolon
	\KwInput{Initial $w_0\in\zset$, gradient estimator $\tilde{g}_{w_0}$, 
	oracle quality $\alpha>0$}
	\Parameter{Step size $\eta$, number of iterations $T$}
	\KwOutput{Point $\bar{w}_T$ satisfying Definition~\ref{def:alphaprox} 
	(for 
	appropriate $\tilde{g}_{w_0}$, $\eta$, $T$)}
	\For{$t = 1, \ldots, T$}
	{
		$w_t \leftarrow \argmin_{w\in\zset}\left\{\inner{\tilde{g}_{w_0}(w_{t - 
		1})}{w} + \frac{\alpha}{2}V_{w_0}(w) + \frac{1}{\eta}V_{w_{t - 1}}(w) 
		\right\}$\; \label{line:inner-step}
	}
	\Return $\bar{w}_T=\frac{1}{T}\sum_{t=1}^T w_t$\; \label{line:average-step}
	\caption{$\InnerLoop(w_0, \tilde{g}_{w_0}, \alpha)$}
\end{algorithm}

\begin{restatable}{proposition}{restateInnerLoop}
\label{prop:innerloopproof}
	Let $\alpha,L>0$, let $w_0\in\zset$ and let $\tilde{g}_{w_0}$ be 
	$(w_0,L)$-centered for monotone $g$. Then, for $\eta = 
	\frac{\alpha}{10L^2}$ and 
	$T \ge  
	\frac{4}{\eta\alpha}=\frac{40L^2}{\alpha^2}$, the iterates of 
	Algorithm~\ref{alg:innerloop} satisfy
	\begin{equation}\label{eq:innerloop-guarantee}
	\Ex{}\max\limits_{u\in\zset}\left[\frac{1}{T} \sum_{t \in [T]} 
	\inner{g(w_t)}{w_t - u} - \alpha V_{w_0}(u)\right]
	\le 0.
	\end{equation}
\end{restatable}

Before discussing the proof of Proposition~\ref{prop:innerloopproof}, we 
state how it implies the relaxed proximal oracle property for the bilinear 
case.
\begin{corollary}\label{cor:innerloop-oracle}
	Let $A\in\R^{m\times n}$ and let $g(z)=(A^\top z\y, -Az\x)$. Then, in 
	the setting of Proposition~\ref{prop:innerloopproof}, 
	$\mc{O}(w_0)=\InnerLoop(w_0, \tilde{g}_{w_0}, \alpha)$ is an 
	$(\alpha,0)$-relaxed proximal oracle.
\end{corollary}
\begin{proof}
	Note that $\inner{g(z)}{w}=-\inner{g(w)}{z}$ for any $z,w\in\zset$ and 
	consequently $\inner{g(z)}{z}=0$. Therefore, the iterates $w_1, \ldots, 
	w_T$ of Algorithm~\ref{alg:innerloop} and its output $\bar{w}_T = 
	\frac{1}{T}\sum_{t=1}^T w_t$ satisfy for every $u\in\zset$,
	\begin{equation*}
	\frac{1}{T}\sum_{t \in [T]} \inner{g(w_t)}{w_t - u} =\frac{1}{T}\sum_{t \in 
	[T]} \inner{g(u)}{w_t}
	=  \inner{g(u)}{\bar{w}_T} =
	\inner{g(\bar{w}_T)}{\bar{w}_T-u}.
	\end{equation*}
	Substituting into the bound~\eqref{eq:innerloop-guarantee}
	yields the $(\alpha, 0)$-relaxed proximal oracle property in 
	Definition~\ref{def:alphaprox}.
\end{proof}
\noindent
More generally, the proof of Corollary~\ref{cor:innerloop-oracle} shows that Algorithm~\ref{alg:innerloop} implements a relaxed 
proximal oracle whenever $z\mapsto \inner{g(z)}{z-u}$ is convex for every 
$u$. In \Cref{ssec:highprecision} we implement an $(\alpha, 
\veps)$-relaxed 
proximal oracle without such an assumption.

The proof of Proposition~\ref{prop:innerloopproof} is a somewhat lengthy  
application of existing techniques for stochastic mirror descent analysis in 
conjunction with Definition~\ref{def:w-l-centered}. We give it in full in 
Appendix~\ref{app:innerloop} 
\notarxiv{
and sketch it briefly here. %
We view Algorithm~\ref{alg:innerloop} as mirror 
descent with stochastic gradients $\tilde{\delta}_t = \tilde{g}_{w_0}(w_t) - 
g(w_0)$ and composite term $\inner{g(w_0)}{z} + 
\tfrac{\alpha}{2}V_{w_0}(z)$. For any $u\in\zset$, the standard mirror 
descent analysis (see 
Lemma~\ref{lem:mirror-descent} in Appendix~\ref{app:mirror-descent}) 
bounds the regret  $\sum_{t\in[T]} \inner{\tilde{g}_{w_0}(w_t) + 
\tfrac{\alpha}{2}\grad V_{w_0}(w_t)}{w_t-u}$ in terms of the distance to 
initialization $V_{w_0}(u)$ and the stochastic gradient norms $ 
\norms{\tilde{\delta}_t}_*^2$ for $t\in[T]$. Bounding these norms 
via~\Cref{def:w-l-centered} 
and rearranging the $\inner{\grad V_{w_0}(w_t)}{w_t-u}$ terms, we show 
that $\E \left[\frac{1}{T} \sum_{t \in [T]} \inner{g(w_t)}{w_t - u} - \alpha 
V_{w_0}(u)\right]\le0$ for all $u\in\zset$. To reach our desired result we 
must swap the order of the expectation and ``for all.'' We do so using the 
``ghost iterate'' technique due to~\citet{NemirovskiJLS09}.
}
\arxiv{
and review the main steps here. 
\paragraph{Regret bound.} Viewing the 
iterations of Algorithm~\ref{alg:innerloop} as stochastic mirror descent 
with stochastic gradients $\tilde{\delta}_t = \tilde{g}_{w_0}(w_t) - g(w_0)$ 
and 
composite term $\inner{g(w_0)}{z} + \tfrac{\alpha}{2}V_{w_0}(z)$, the 
standard 
mirror descent regret bound (see Lemma~\ref{lem:mirror-descent} in 
Appendix~\ref{app:mirror-descent}) gives
\begin{equation}\label{eq:composite-regret-basic}
\sum_{t\in[T]} \inner{\tilde{g}_{w_0}(w_t) + \tfrac{\alpha}{2}\grad 
V_{w_0}(w_t)}{w_t-u} \le \frac{V_{w_0}(u)}{\eta} 
+ \frac{\eta}{2}\sum_{t\in[T]} \norms{\tilde{\delta}_t}_*^2
\end{equation}
deterministically for all $u\in\zset$. 

\paragraph{Regularization.} Substituting the equality~\eqref{eq:three-point} and rearranging gives
\begin{equation}\label{eq:stoch-reg-bound}
\frac{1}{T}
\sum_{t\in[T]} 
\inner{\tilde{g}_{w_0}(w_t)}{w_t - u} - \alpha V_{w_0}(u)
\le \left(\frac{1}{\eta T} - \frac{\alpha}{2}\right) V_{w_0}(u) 
+ \frac{1}{T}\sum_{t\in[T]}
	 \left[ \frac{\eta}{2} \norms{\tilde{\delta}_t}_*^2 
	 - \frac{\alpha}{2} V_{w_0}(w_t)
	 \right]
\end{equation}
and taking $T \ge \frac{4}{\eta \alpha}$ guarantees $\big(\frac{1}{\eta T} - 
\frac{\alpha}{2}\big) V_{w_0}(u) \le 0$ for all $u$. 

\paragraph{Variance bound.} 
Using the second centered gradient estimator property and strong 
convexity 
of the distance generating function, we have
\begin{equation*}
\E \left[ \frac{\eta}{2} \norms{\tilde{\delta}_t}_*^2 
- \frac{\alpha}{2} V_{w_0}(w_t)
\right] \le 
\E \left[ \frac{\eta L^2}{2} \norm{w_t - w_0}^2 
- \frac{\alpha}{2} V_{w_0}(w_t)
\right] \le 
\left(\eta L^2- \frac{\alpha}{2} \right) \E V_{w_0}(w_t) \le 0
\end{equation*}
for $\eta \le \frac{\alpha}{2L^2}$. Since the RHS 
of~\eqref{eq:stoch-reg-bound} is nonpositive in expectation and the 
gradient 
estimator is unbiased, we have 
$\max\limits_{u\in\zset}\Ex{}\left[\frac{1}{T} \sum_{t \in [T]} 
\inner{g(w_t)}{w_t - u} - \alpha V_{w_0}(u)\right]
\le 0.$ 

\paragraph{Exchanging maximum and expectation.} When $u$ depends 
on $\tilde{g}_{w_0}(w_t)$ we generally have \linebreak
$\E \inner{\tilde{g}_{w_0}(w_t)-g(w_t)}{w_t-u}\ne 0$. 
To address this issue we use a technique due to~\citet{NemirovskiJLS09}. 
Writing $\tilde{\Delta}_t = \tilde{g}_{w_0}(w_t)-g(w_t)$ and defining the  
``ghost 
iterates'' $s_t = \prox{s_{t-1}}{\tilde{\Delta}_{t-1}}[1/\eta]$ 
 with $s_0 = 
w_0$, we rewrite $\inners{\tilde{\Delta}_t}{w_t-u}$ as 
$\inners{\tilde{\Delta}_t}{w_t-s_t} + \inners{\tilde{\Delta}_t}{s_t-u}$. Since 
$s_t$ does not depend on randomness in $\tilde{g}_{w_0}(w_t)$, we have 
$\E \inners{\tilde{\Delta}_t}{w_t-s_t}=0$. To handle the term 
$\sum_{t}\inners{\tilde{\Delta}_t}{s_t-u}$ we use the standard mirror 
descent regret bound again, absorbing the result into the RHS 
of~\eqref{eq:stoch-reg-bound} using 
$V_{s_0}(u)=V_{w_0}(u)$ and  $\E 
\norms{\tilde{\Delta}_t}_*^2 \le 
4L^2 \E 
\norm{w_t - w_0}^2$, which follows from Lemma~\ref{lem:centered-var}.
}

\section{Application to bilinear saddle point problems}
\label{sec:apps}

We now construct centered gradient estimators (as per 
Definition~\ref{def:w-l-centered}) for the linear gradient mapping 
\begin{equation*}
g(z) =  (A^\top z\y, -Az\x) 
\text{ corresponding to the bilinear saddle point problem }
\min_{x\in\xset}\max_{y\in\yset} y^\top A x.
\end{equation*}
\arxiv{
We consider two domain types, namely $\ell_1$ (the simplex) and $\ell_2$ 
(the Euclidean ball). In Section~\ref{ssec:l1l1} we present a centered 
gradient estimator and resulting runtime guarantees for $\ell_1$-$\ell_1$ 
games. In Section~\ref{ssec:l1l2} we first give a centered gradient estimator 
$\ell_2$-$\ell_1$ with a suboptimal constant $L$ (larger than the Lipschitz 
constant of $g$). We then obtain the correct Lipschitz constant dependence 
using a local norms analysis, which requires clipping the gradient estimates 
in order to control the magnitude of the updates. Finally, in 
Section~\ref{ssec:l2l2} we give a gradient estimator for $\ell_2$-$\ell_2$ 
games. Unlike the previous two setups, the estimator constant $L$ for 
$\ell_2$-$\ell_2$ games does not match the Lipschitz constant of the 
underlying gradient mapping. Such mismatch is consistent with prior 
findings in the literature.
}%
\notarxiv{
	\Cref{ssec:l1l1,ssec:l1l2} consider the $\ell_1$-$\ell_1$ and 
	$\ell_2$-$\ell_1$ settings, respectively; in \Cref{ssec:l2l2} 
	we show how our approach naturally extends to the $\ell_2$-$\ell_2$ 
	setting as well.
}%
Throughout, we let $w_0$ denote the ``center'' (i.e.\ reference 
point) of our stochastic gradient estimator and consider a general query 
point $w\in\zset=\xset\times\yset$. We also recall the notation $[v]_i$ for 
the $i$th entry of vector $v$. %

\subsection{$\ell_1$-$\ell_1$ games}
\label{ssec:l1l1}

\paragraph{Setup.}
Denoting the $d$-dimensional simplex by $\Delta^d$, we let  
$\xset=\Delta^n$, $\yset=\Delta^m$ and $\zset=\xset\times\yset$. 
We take $\norm{z}$ to be $\sqrt{\lone{z\x}^2 + \lone{z\y}^2}$ with 
dual norm
$\norm{\gamma}_* = \sqrt{\linf{\gamma\x}^2 + \linf{\gamma\y}^2}$. We 
take the distance 
generating function $r$ to be the negative entropy, i.e.\ $r(z)=\sum_i [z]_i 
\log [z]_i$, and note that it is 1-strongly convex w.r.t.\ $\norm{\cdot}$ 
and has range $\Theta = \log(mn)$.
Finally we set 
\[\norm{A}_{\max} \defeq \max_{i, j}|A_{ij}|\] and note that this is the 
 Lipschitz constant of the gradient mapping $g$ under the chosen norm.

\arxiv{
\subsubsection{Gradient estimator}
}
\notarxiv{
\paragraph{Gradient estimator.}
}
Given $w_0=(w_0\x, w_0\y)$ and  $g(w_0) = (A^\top w_0\y, -A w_0\x)$, 
we describe the reduced-variance gradient estimator $\tilde{g}_{w_0}(w)$. 
First, we define the probabilities $p(w)\in\Delta^m$ and $q(w)\in \Delta^n$ 
according to,
\begin{equation}\label{eq:l1-prob-def}
p_i(w)\defeq\frac{\left|[w\y]_i-[w_0\y]_i\right|}{\left\Vert w\y - 
w_0\y\right\Vert _{1}}~~\mbox{and}~~\ 
q_j(w)\defeq\frac{|[w\x]_j-[w_0\x]_j|}{\left\Vert w\x - w_0\x\right\Vert 
_{1}}.
\end{equation}
To compute  $\tilde{g}_{w_0}$ we sample $i\sim p(w)$ and $j\sim q(w)$ 
independently, and set
\begin{equation}
\label{eq:tgdef}
\begin{aligned}
\tilde{g}_{w_0}(w)& \defeq \left(A^\top 
w_0\y+\ai\frac{[w\y]_i-[w_0\y]_i}{p_i(w)},
-Aw_0\x-\aj\frac{[w\x]_j-[w_0\x]_j}{q_j(w)}\right),
\end{aligned}
\end{equation}
where $\ai$ and $\aj$ are the $i$th row and $j$th column of $A$, 
respectively. 
Since the sampling distributions $p(w),q(w)$ are proportional to the absolute 
value of the difference between blocks of $w$ and $w_0$, we call
strategy~\eqref{eq:l1-prob-def} ``sampling from the difference.'' 
Substituting~\eqref{eq:l1-prob-def} into~\eqref{eq:tgdef} gives the 
explicit form %
\begin{equation}
\label{eq:tgdef-l1}
\tilde{g}_{w_0}(w) = g(w_0) + \left(
\ai  {\lones{w\y - w_0\y}}{\sign([w\y-w_0\y]_i)}, 
-\aj  {\lones{w\x - w_0\x}}{\sign([w\x-w_0\x]_j)}\right).
\end{equation}
A straightforward calculation 
shows that this construction satisfies Definition~\ref{def:w-l-centered}.

\begin{restatable}{lemma}{restateSimplexEst}
\label{lem:l1-gradient-est} 
In the $\ell_1$-$\ell_1$ setup, the estimator~\eqref{eq:tgdef-l1} is 
$(w_0,L)$-centered with $L=\norm{A}_{\max}$.
\end{restatable}

\notarxiv{\vspace{-\baselineskip}}

\begin{proof}
	The first property ($\E \tilde{g}_{w_0}(w) = g(w)$) follows immediately by 
	inspection of~\eqref{eq:tgdef}. The second property follows 
	from~\eqref{eq:tgdef-l1} by noting that
	\begin{equation*}
	\norm{\tilde{g}_{w_0}(w) - g(w_0)}_*^2 
	=  
	\linf{\ai}^2 \lones{w\y -w_0\y}^2 + 
	\linf{\aj}^2 \lones{w\x -w_0\x}^2  
	\le \norm{A}_{\max}^2\norm{w - w_0}^2.
	\end{equation*}
	for all $i,j$, and therefore $\E \norm{\tilde{g}_{w_0}(w) - g(w_0)}_*^2 
	\le \norm{A}_{\max}^2\norm{w - actw_0}^2$. 
\end{proof}

\noindent
The proof of Lemma~\ref{lem:l1-gradient-est} reveals that the proposed 
estimator satisfies a stronger version of Definition~\ref{def:w-l-centered}: 
the last property and also Lemma~\ref{lem:centered-var} hold with 
probability 1 rather than in expectation.

\arxiv{
We note that while it naturally arises from our variance requirements, our 
gradient estimator appears to be fundamentally different from those 
used in known variance-reduced 
algorithms~\citep{BalamuruganB16,ShiZY17,ChavdarovaGFL19,MishchenkoKSRM19}.
In particular, in standard 
finite-sum settings, estimators in the literature sample from a fixed 
distribution~\citep{Johnson13, Zhu17, BottouCN18}. In contrast, our 
sampling 
distributions change dynamically with respect to the current point $w$, 
similarly to the (fixed-variance) estimators in~\cite{ClarksonHW10}. 
}

\newcommand{\algOneOne}{
\SetKwComment{Comment}{$\triangleright$\ }{}
\SetCommentSty{color{black}}
\begin{algorithm}[t]
	\DontPrintSemicolon
	\KwInput{Matrix $A\in\R^{m\times n}$ with $i$th row $\ai$ and $j$th 
		column $\aj$, target accuracy $\epsilon$}
	\KwOutput{A point with expected duality gap below $\epsilon$}
	
	\vspace{3pt}
	$L\gets\max_{ij}|A_{ij}|$, $\alpha \gets L\sqrt{\frac{n+m}{\nnz(A)}}$, 
	$K\gets \ceil{\frac{\log(nm)\alpha}{\epsilon}}$, $\eta \gets 
	\frac{\alpha}{10L^2}$, $T\gets\ceil{\frac{4}{\eta\alpha}}$, 
	$z_0\gets (\frac{1}{n}\textbf{1}_n, \frac{1}{m}\textbf{1}_m)$\;
	\For{$k=1,\ldots,K$}
	{
		\vspace{3pt}%
		\Comment{\emph{Relaxed oracle query:}}

		$\displaystyle (x_0, y_0) \gets (z_{k-1}\x,z_{k-1}\y)$, $(g\x_0, g\y_0) 
		\gets (A^\top y_0, -A x_0)$\label{line:init-center}\; 
		
		\vspace{3pt}
		
		\For{$t=1,\ldots,T$}
		{
			\vspace{3pt}%
			\Comment{\emph{Gradient estimation:}}
			\vspace{-4pt}%
			
			Sample $i\sim p$ where $\displaystyle p_i = \frac{\left| [y_{t-1}]_i 
				- [y_{0}]_i 
				\right|}{\lone{y_{t-1}-y_0}}$,
			sample $j\sim q$ where $\displaystyle q_j = \frac{\left| [x_{t-1}]_j 
				- [x_{0}]_j 
				\right|}{\lone{x_{t-1}-x_0}}$\;
			
			Set $\displaystyle\tilde{g}_{t-1} = g_0 +  \left(
			\ai \frac{[y_{t-1}]_i - [y_{0}]_i}{p_i}, 
			-\aj \frac{[x_{t-1}]_j - [x_{0}]_j}{q_j}\right)$\;

			\Comment{\emph{Mirror descent step:}}
			
			$\displaystyle x_t \gets \Pi_{\xset}\left( \frac{1}{1+\eta\alpha/2} 
			\left( \log x_{t 
				- 1} + \frac{\eta\alpha}{2}\log x_0 - \eta \tilde{g}_{t-1}\x 
			\right)\right)$
			\Comment*[f]{$\Pi_{\xset}(v)=\frac{e^v}{\lone{e^v}}$}
			\;

			$\displaystyle y_t \gets \Pi_{\yset}\left( \frac{1}{1+\eta\alpha/2} 
			\left(
			\log y_{t - 1} + 
			\frac{\eta\alpha}{2}\log 
			y_0 - \eta \tilde{g}_{t-1}\y
			\right)\right)$
			\Comment*[f]{$\Pi_{\yset}(v)=\frac{e^v}{\lone{e^v}}$}
			\; 
		}
		$\displaystyle z_{k - 1/2} \gets \frac{1}{T}\sum_{t=1}^T (x_t, y_t)$\;

		\Comment{\emph{Extragradient step:}}
		
		$\displaystyle z_k\x \gets \Pi_{\xset}\left(\log z_{k-1}\x 
		-\tfrac{1}{\alpha} A^\top z_{k-1/2}\y 
		\right)$\label{line:extragrad-x}\;

		$\displaystyle z_k\y \gets \Pi_{\yset}\left(\log z_{k-1}\y 
		+\tfrac{1}{\alpha} A z_{k-1/2}\x \right)$\label{line:extragrad-y}\;

	}
	\Return $\displaystyle \frac{1}{K}\sum_{k=1}^K z_{k-1/2}$
    \caption{Variance reduction for $\ell_1$-$\ell_1$ 
    games}
	\label{alg:l1l1}
\end{algorithm}
}

\arxiv{\algOneOne}

\arxiv{
	\subsubsection{Full algorithm and complexity 
		analysis}
}
\notarxiv{
	\paragraph{Runtime bound.}
}
Combining the centered gradient estimator~\eqref{eq:tgdef}, the relaxed 
oracle implementation (Algorithm~\ref{alg:innerloop}) and the extragradient
outer loop (Algorithm~\ref{alg:outerloop}), we obtain our main result for 
$\ell_1$-$\ell_1$ games: an accelerated stochastic variance reduction 
algorithm. We write the resulting complete method explicitly as 
Algorithm~\ref{alg:l1l1}\notarxiv{ in~\Cref{app:l1l1-alg}}. The algorithm 
enjoys the following runtime guarantee\notarxiv{ (see proof 
	in~\Cref{app:l1l1-proof})}.

\begin{restatable}{theorem}{restateThmOneOne}
\label{thm:l1l1}
Let $A\in\R^{m\times n}$, $\epsilon > 0$, and $\alpha \geq 
\epsilon / \log(nm)$. Algorithm~\ref{alg:l1l1} outputs a point 
$z=(z\x,z\y)$ such that
\arxiv{
\begin{equation*}
\E\, \left[\max_{y\in\Delta^m} y^\top A z\x - \min_{x\in\Delta^n} 
(z\y)^\top A x \right]
=
\E \left[\max_{i}\, {[Az\x]}_i - \min_{j} \, {[A^\top 
z\y]}_j\right] \le \epsilon,
\end{equation*}
}%
\notarxiv{
$\E\, \left[\max_{y\in\Delta^m} y^\top A z\x - \min_{x\in\Delta^n} 
(z\y)^\top A x \right]
=
\E \left[\max_{i}\, {[Az\x]}_i - \min_{j} \, {[A^\top 
	z\y]}_j\right] \le \epsilon$,
}
 and runs in time 
\begin{equation}\label{eq:l1l1-time}
O\left(\left(\nnz(A) + \frac{(m + n)\norm{A}_{\max}^2}{\alpha^2} 
\right)\frac{\alpha\log(mn)}{\epsilon}\right).
\end{equation}
Setting $\alpha$ optimally, the running time is 
\begin{equation}\label{eq:l1l1-time-opt}
O\left(\nnz(A) + \frac{\sqrt{\nnz(A) (m + n)} 
\norm{A}_{\max}\log(mn)}{\epsilon}\right).
\end{equation}
\end{restatable}
\newcommand{\proofOneOne}
{
\begin{proof}
First, we prove the expected duality gap bound. By 
Lemma~\ref{lem:l1-gradient-est} and 
Corollary~\ref{cor:innerloop-oracle} (with $L = \norm{A}_{\max}$), 
$\InnerLoop$ is an ($\alpha,0$)-relaxed proximal oracle. On 
$\Delta^d$, negative entropy has minimum value $-\log d$ and is non-positive, therefore 
for the $\ell_1$-$\ell_1$ domain we have $\Theta = \max_{z'}r(z') - 
\min_{z}r(z) = \log(nm)$. By Proposition~\ref{prop:outerloopproof}, running 
$K\ge \alpha \log(nm) /\epsilon$ iterations 
guarantees an $\epsilon$-approximate saddle 
point in expectation. 

Now, we prove the runtime bound. Lines~\ref{line:init-center}, 
\ref{line:extragrad-x} and \ref{line:extragrad-y} of Algorithm~\ref{alg:l1l1} 
each take time $O(\nnz(A))$, as they involve matrix-vector products with 
$A$ and $A^\top$. All other lines run in time $O(n+m)$, 
as they consist of sampling and vector arithmetic (the time to compute 
sampling probabilities dominates the runtime of sampling).
 Therefore, the total runtime 
is $O((\nnz(A) + (n+m)T)K)$. Substituting $T\le 1+\frac{40L^2}{\alpha^2}$ 
and $K\le1+\frac{\log(nm)\alpha}{\epsilon}$ gives the 
bound~\eqref{eq:l1l1-time}.
Setting
\begin{equation*}
\alpha=\max\left\{
\frac{\epsilon}{\log{nm}}\,,\,
\norm{A}_{\max} \sqrt{\frac{n+m}{\nnz(A)}}
\right\}
\end{equation*}
gives the optimized bound~\eqref{eq:l1l1-time-opt}.
\end{proof}

\begin{remark}
We can improve the $\log(mn)$ factor in~\eqref{eq:l1l1-time} 
and~\eqref{eq:l1l1-time-opt} to $\sqrt{\log 
m \log n}$ by the transformation $\xset \to c\cdot \xset$ and 
$\yset \to \frac{1}{c} \cdot \yset$ where $c=\left({\log m}/{\log 
	n}\right)^{1/4}$. This 
transformation leaves the 
problem unchanged and reduces $\Theta$ 
from $\log(mn)$ to $2\sqrt{\log m \log n}$. It is also 
equivalent to proportionally using slightly different step-sizes for the $\xset$ and 
$\yset$ block. See also~\cite[Example 1]{Nemirovski04} and 
Section~\ref{ssec:strongly}.
\end{remark}
}
\arxiv{\proofOneOne}

\subsection{$\ell_2$-$\ell_1$ games}
\label{ssec:l1l2}
\paragraph{Setup.}
We set $\xset=\ball^n$ to be the
$n$-dimensional Euclidean ball of radius $1$, while $\yset=\Delta^m$ 
remains the simplex. For $z = (z\x, z\y)\in\zset=\xset\times\yset$ we 
define a norm by 
\[
\norm{z}^{2}=\norm{z\x}_{2}^{2}+\norm{z\y}_{1}^{2}
~\text{ with dual norm }~
 \norm{g}_*^2=\norm{g\x}_{2}^{2}+\norm{g\y}_{\infty}^{2}.
\]
For distance generating function we take $r(z)=r\x(z\x)+r\y(z\y)$ with 
$r\x(x)=\frac{1}{2}\norm{x}_2^2$ and $r\y(y)=\sum_i y_i \log y_i$; $r$ is 
1-strongly convex w.r.t.\ to $\norm{\cdot}$ and has range $\half+\log m 
\le 
\log(2m)$. 
 Finally, we denote  
 \[
 \nA= \max_{i \in [m]} \norm{\ai}_2,
 \]
 and note that this is the Lipschitz constant of $g$ under $\norm{\cdot}$.

\newcommand{\presentationTwoOne}{
\subsubsection{Basic gradient estimator}\label{ssec:l2l1-basic}
We first present a straightforward adaptation of the $\ell_1$-$\ell_1$ 
gradient estimator, which we subsequently improve to obtain the optimal 
Lipschitz constant dependence. Following the ``sampling from the 
difference'' strategy, consider a gradient estimator $\tilde{g}_{w_0}$ 
computed as in~\eqref{eq:tgdef}, but with the following different choice of 
$q(w)$:
\begin{equation}
\label{eq:tgdef-l2-probs}
\begin{aligned}
 p_i(w)=\frac{\left|[w\y]_i-[w_0\y]_i\right|}{\left\Vert w\y - 
 w_0\y\right\Vert _{1}}~~\mbox{and}~~
 \ q_j(w)=\frac{([w\x]_j-[w_0\x]_j)^2}{\left\Vert w\x - 
 w_0\x\right\Vert^2_{2}}.
\end{aligned}
\end{equation}
The resulting gradient estimator has the explicit form
\begin{equation}
\label{eq:tgdef-l2}
\tilde{g}_{w_0}(w) = g(w_0) + \left(
\ai \frac{\lones{w\y - w_0\y}}{\sign([w\y-w_0\y]_i) }, 
-\aj  \frac{\ltwos{w\x - w_0\x}^2}{[w\x-w_0\x]_j}\right).
\end{equation}
(Note that $\tilde{g}_{w_0}$ of the form~\eqref{eq:tgdef} is finite with 
probability 1.) Direct calculation shows it is centered.

\begin{restatable}{lemma}{restateNaive}
\label{lem:l2-gradient-est-naive}
In the $\ell_2$-$\ell_1$ setup, the estimator~\eqref{eq:tgdef-l2} is 
$(w_0,L)$-centered with $L = 
\sqrt{\sum_{j\in[n]}\|\aj\|^2_\infty}$.
\end{restatable}
\begin{proof}
	The estimator is unbiased since it is of the form~\eqref{eq:tgdef}. To 
	show the variance bound, first consider the $\xset$-block. We have
	\begin{equation}\label{eq:l2l1-centered-x}
	\ltwo{\tilde{g}_{w_0}\x(w) - g\x(w_0)}^2 
	= 
	\ltwo{\ai}^2\lones{w\y-w_0\y}^2
	\le
	\nA^2 \lones{w\y-w_0\y}^2
	\le 
	L^2 \lones{w\y-w_0\y}^2,
	\end{equation}
	where we used $\nA^2 = \max_{i\in [n]} \ltwo{\ai}^2 \le 
	\sum_{j\in[m]} \linf{\aj}^2 = L^2$. Second, for the $\yset$-block,
	\begin{equation}\label{eq:l2l1-centered-naive-y}
	\E \linf{\tilde{g}_{w_0}\y(w) - g\y(w_0)}^2 
	= \sum_{j\in [n]}\frac{\linf{\aj}^2[w\x-w_0\x]_j^2}{q_j(w)} 
	= L^2 \ltwo{w\x-w_0\x}^2.
	\end{equation}
	Combining~\eqref{eq:l2l1-centered-x} 
	and~\eqref{eq:l2l1-centered-naive-y}, we have the second property 
	$\E \norm{\tilde{g}_{w_0}(w) - g(w_0)}_*^2 \le L^2\norm{w-w_0}^2$. 
\end{proof}

\subsubsection{Improved gradient estimator}\label{ssec:l2l1-clip}

The constant $L$ in Lemma~\ref{lem:l2-gradient-est-naive} is larger than 
the Lipschitz constant of $g$ (i.e.\ $\nA$) by a factor of up to $\sqrt{n}$. 
Consequently, a variance reduction scheme based on the 
estimator~\eqref{eq:tgdef-l2} will not always improve on the linear-time 
mirror prox method. %

 Inspecting the proof of 
Lemma~\ref{lem:l2-gradient-est-naive}, we see that the cause for the 
inflated value of $L$ is the bound~\eqref{eq:l2l1-centered-naive-y} on $\E 
\linf{\tilde{g}_{w_0}\y(w) - g\y(w_0)}^2$. We observe that swapping the 
order of expectation and maximization would solve the problem, as 
\begin{equation}\label{eq:l2l1-centered-y-swap}
\max_{k \in [m]} \E \, [\tilde{g}_{w_0}\y(w) - g\y(w_0)]_k^2 
= \max_{k \in [m]}  \sum_{j\in [n]}\frac{A_{kj}^2[w\x-w_0\x]_j^2}{q_j(w)} 
=  \nA^2 \ltwo{w\x-w_0\x}^2.
\end{equation}
Moreover, inspecting the proof 
of Proposition~\ref{prop:innerloopproof} reveals that instead of bounding 
terms of the form
$\E 
\linf{\tilde{g}_{w_0}\y(w_t) - g\y(w_0)}^2$ we may directly bound 
$\E \left[\eta \inners{\tilde{g}_{w_0}\y(w_{t}) - g\y(w_0)}{y_{t}-y_{t+1}} - 
V_{y_{t}}(y_{t+1})\right]$, where we write $w_t = (x_t, y_t)$ and recall that 
$\eta$ is the step-size in Algorithm~\ref{alg:innerloop}. Suppose that 
$\eta\linf{\tilde{g}_{w_0}\y(w_{t}) - g\y(w_0)}\le 1$ holds. In this case 
we 
may use a ``local norms'' bound (Lemma~\ref{lem:local-norms-classical} 
in 
Appendix~\ref{app:entropy-props}) to write
\begin{equation*}
\eta \inners{\tilde{g}_{w_0}\y(w_t) - g\y(w_0)}{y_{t}-y_{t+1}} 
- V_{y_{t}}(y_{t+1})
\le 
\eta^2 \sum_{k\in[m]} [y_{t}]_k
[\tilde{g}_{w_0}\y(w_t) - g\y(w_0)]_k^2
\end{equation*}
and bound the expectation of the RHS 
using~\eqref{eq:l2l1-centered-y-swap} conditional on $w_t$. 

Unfortunately, the gradient estimator~\eqref{eq:tgdef-l2} does not always 
satisfy 
 $\eta\linf{\tilde{g}_{w_0}\y(w_{t}) - g\y(w_0)}\le 1$. 
Following~\citet{ClarksonHW10}, we enforce this bound by 
clipping the gradient estimates, yielding the estimator
\begin{equation}
\label{eq:tgdef-l2-dynamic}
\begin{aligned}
\tilde{g}_{w_0}(w)& \defeq 
\left(
A^\top w_0\y+\ai\frac{[w\y]_i-[w_0\y]_i}{p_i(w)},
-Aw_0\x-\trunc\left(\aj\frac{[w\x]_j-[w_0\x]_j}{q_j(w)}\right)
\right),\\
& \text{where }\left[\trunc\left(v\right)\right]_{i}=\begin{cases}
-\tau & [v]_{i}<-\tau\\
[v]_{i} & -\tau\le [v]_i\le\tau\\
\tau & [v]_{i}>\tau,
\end{cases}
\end{aligned}
\end{equation}
where $i\sim p(w)$ and $j\sim q(w)$ with $p,q$ as defined 
in~\eqref{eq:tgdef-l2-probs}. The clipping in~\eqref{eq:tgdef-l2-dynamic} 
does not significantly change the variance of the estimator, but it 
introduces 
some bias for which we must account.
We summarize the relevant properties of the clipped gradient estimator in 
the following.

\begin{restatable}{definition}{restateDefCBB}
	\label{def:z-l-cbb}
	Let $w_0=(w_0\x,w_0\y)\in\zset$ and $\tau,L>0$. 
	A stochastic gradient estimator 
	$\tilde{g}_{w_0}:\zset\rightarrow\zset^*$ 
	is called \emph{$(w_0,L,\tau)$-centered-bounded-biased (CBB)} if it satisfies for 
	all 
	$w=(w\x,w\y)\in\zset$,
	\begin{enumerate}
	\item $\E\tilde{g}\x_{w_0}(w) = g\x(w)$ and $\norm{\E 
	\tilde{g}\y_{w_0}(w) - g\y(w)}_* \le 
	\frac{L^2}{\tau}\norm{w-w_0}^2$,
	\item $\norm{\tilde{g}_{w_0}\y(w) - g\y(w_0)} _* \le \tau$,
	\item 
	$\E \norm{\tilde{g}_{w_0}\x(w) - 
		g\x(w_0)}_*^2 + \max_{i\in[m]} \E \, [\tilde{g}_{w_0}\y(w) - 
		g\y(w_0)]_i^2 \le L^2 
	\norm{w - w_0}^2$.
	\end{enumerate}
\end{restatable}
\begin{restatable}{lemma}{restateClipped}
	\label{lem:l2-gradient-est-dynamic}
	In the $\ell_2$-$\ell_1$ setup, the 
	estimator~\eqref{eq:tgdef-l2-dynamic} is $(w_0,L,\tau)$-CBB with $L = 
	\nA$.
\end{restatable}
\begin{proof}
	The $\xset$ component for the gradient estimator is unbiased.
	We bound the bias in the 
	$\yset$ block as follows.
	Fixing an index $i\in[m]$, we have
	\begin{align*}
	\left|\Ex{}\left[ 
	\tilde{g}_{w_0}\y\left(w\right)-g\y\left(w\right)\right]_{i}\right|
	& 
	=\left|\E_{j}\left[ A_{ij}\frac{[w\x]_j-[w_0\x]_j}{q_{j}} 
	-\trunc\left(A_{ij}\frac{[w\x]_j-[w_0\x]_j}{q_{j}}\right)\right]\right|\\
	& 
	\le\sum_{j\in\mathcal{J}_{\tau}\left(i\right)} q_{j} 
	\left|A_{ij}\frac{[w\x]_j-[w_0\x]_j}{q_{j}} 
	-\trunc\left(A_{ij}\frac{[w\x]_j-[w_0\x]_j}{q_{j}}\right)\right|\\
	& 
	\le\sum_{j\in\mathcal{J}_{\tau}\left(i\right)}\left|A_{ij}\right|\left|[w\x]_j-[w_0\x]_j\right|
	\end{align*}
	where the last transition used $\left|a-\trunc (a)\right| \le  
	\left|a\right|$ 
	for all $a$, and
	\[
	\mathcal{J}_{\tau}\left(i\right)=\left\{ j\in[n]\mid 
	\trunc\left(A_{ij}\frac{[w\x]_j-[w_0\x]_j}{q_{j}}\right)\ne 
	A_{ij}\frac{[w\x]_j-[w_0\x]_j}{q_{j}}\right\} .
	\]
	Note that $j\in\mathcal{J}_{\tau}\left(i\right)$ if and only if
	\[
	\left|A_{ij}\frac{[w\x]_j-[w_0\x]_j}{q_{j}}\right| = \frac{\left\Vert 
		w\x-w_0\x\right\Vert 
		_{2}^{2}\left|A_{ij}\right|}{\left|[w\x]_j-[w_0\x]_j\right|}>\tau\Rightarrow\left|[w\x]_j-[w_0\x]_j\right|\le\frac{1}{\tau}\left\Vert
	w\x-w_0\x\right\Vert _{2}^{2}\left|A_{ij}\right|.
	\]
	Therefore,
	\[
	\sum_{j\in\mathcal{J}_{\tau}\left(i\right)}\left|A_{ij}\right|\left|[w\x]_j-[w_0\x]_j\right|\le\frac{1}{\tau}\left\Vert
	w\x-w_0\x\right\Vert 
	_{2}^{2}\sum_{j\in\mathcal{J}_{\tau}\left(i\right)}\left|A_{ij}\right|^{2} \le 
	\frac{1}{\tau}\left\Vert
	w\x-w_0\x\right\Vert _{2}^{2}\left\Vert \ai\right\Vert_2 ^{2}
	\]
	and $\linf{\E \tilde{g}_{w_0}\y(w) - g\y(w)} \le 
	\frac{L^2}{\tau}\ltwo{w\x-w_0\x}^2$ follows by taking the maximum 
	over $i\in[m]$. 
	
	The property $\linf{\tilde{g}_{w_0}\y(w)-g\y(w_0)}\le \tau$ follows 
	directly from the definition of $\trunc$.
	Finally, we note that for all $k$, the addition of $\trunc$ never increases  
	$[\tilde{g}_{w_0}\y(w) - g\y(w_0)]_k^2$, and so the third property 
	follows 
	from~\eqref{eq:l2l1-centered-y-swap} and~\eqref{eq:l2l1-centered-x}.
\end{proof}

To guarantee $\eta\linf{\tilde{g}_{w_0}\y(w_{t}) - g\y(w_0)}\le 1$, we 
set the threshold $\tau$ to be $1/\eta$. By the first property 
in Definition~\ref{def:z-l-cbb}, the bias caused by this choice of $\tau$ is 
of the order of the variance of the estimator, and we may therefore cancel it 
 with the regularizer by choosing $\eta$ slightly smaller than in 
Proposition~\ref{prop:innerloopproof}. In 
Appendix~\ref{app:cbb} we prove (using the observations from the 
preceding discussion) that Algorithm~\ref{alg:innerloop} with a CBB 
gradient estimator implements a relaxed proximal oracle.
\begin{restatable}{proposition}{restateInnerLoopCBB}
	\label{prop:innerloop-cbb}
	In the $\ell_2$-$\ell_1$ setup, let $\alpha,L>0$, let $w_0\in\zset$ and 
	let $\tilde{g}_{w_0}$ be 
	$(w_0,L,\frac{24 L^2}{\alpha})$-CBB for monotone and $L$-Lipschitz 
	$g$. 
	Then, for  
	$\eta = 
	\frac{\alpha}{24 L^2}$ and 
	$T \ge  
	\frac{4}{\eta\alpha}=\frac{96 L^2}{\alpha^2}$, the iterates of 
	Algorithm~\ref{alg:innerloop} satisfy the 
	bound~\eqref{eq:innerloop-guarantee}. Moreover, for $g(z)=(A^\top z\y, 
	-Az\x)$, 
	$\mc{O}(w_0)=\InnerLoop(w_0, \tilde{g}_{w_0}, \alpha)$ is an 
	$(\alpha,0)$-relaxed proximal oracle.
\end{restatable} 
\noindent
We remark that the proof of Proposition~\ref{prop:innerloop-cbb} relies on 
the structure of the simplex with negative entropy as the distance 
generating function. For this reason, we state the proposition for the 
$\ell_2$-$\ell_1$ setup. However, Proposition~\ref{prop:innerloop-cbb} 
would also hold for other setups where $\yset$ is the simplex and $r\y$ is 
the negative entropy, provided a CBB gradient estimator is available.
}

\arxiv{\presentationTwoOne}
\notarxiv{
	\paragraph{Gradient estimator.}
	To account for the fact that $\xset$ is now the $\ell_2$ unit ball, we 
	modify the sampling distribution $q$ in~\eqref{eq:l1-prob-def} to 
	$q_j(w)=\frac{([w\x]_j-[w_0\x]_j)^2}{\left\Vert w\x - 
		w_0\x\right\Vert^2_{2}}$, and keep $p$ the same. 
	As we explain in detail in~\Cref{ssec:l2l1-basic}, substituting these 
	probabilities into the expression~\eqref{eq:tgdef} yields a centered 
	gradient estimator with a constant $(\sum_{j\in[n]} \linf{\aj}^2)^{1/2}$ 
	that is larger than $\nA$ by a factor of up to $\sqrt{n}$. Using local 
	norms analysis allows us to tighten these bounds whenever the 
	stochastic steps have bounded infinity norm. 
	Following~\citet{ClarksonHW10}, we enforce such bound on the step 
	norms via gradient clipping. 
	The final gradient estimator is
	\begin{equation*}
	\begin{aligned}
	\tilde{g}_{w_0}(w)& \defeq 
	\left(
	A^\top w_0\y+\ai  \frac{\lones{w\y - w_0\y}}{\sign([w\y-w_0\y]_i)},
	-Aw_0\x-\trunc\left(\aj\frac{\ltwo{w\x-w_0\x}^2}{[w\x]_j-[w_0\x]_j}\right)
	\right),\\
	& \text{where }\left[\trunc\left(v\right)\right]_{i}=\begin{cases}
	-\tau & [v]_{i}<-\tau\\
	[v]_{i} & -\tau\le [v]_i\le\tau\\
	\tau & [v]_{i}>\tau,
	\end{cases}
	\end{aligned}
	\end{equation*}
	
	The clipping operation $\trunc$ introduces bias to the gradient 
	estimator, which we account for by carefully choosing a value of $\tau$ 
	for which the bias is on the same order as the variance, and yet the 
	resulting steps are appropriately bounded; 
	see~\Cref{ssec:l2l1-clip}. In~\Cref{ssec:l1l2-add} we describe an 
	alternative gradient 
	estimator for which the distribution $q$ does not depend on the current 
	iterate $w$.
}

\arxiv{
\subsubsection{Full algorithm and complexity analysis}
}
\notarxiv{
\paragraph{Runtime bound.}
}

\newcommand{\algTwoOne}{
\begin{algorithm} %
	\DontPrintSemicolon
	\KwInput{Matrix $A\in\R^{m\times n}$ with $i$th row $\ai$ and $j$th 
		column $\aj$, target accuracy $\epsilon$}
	\KwOutput{A point with expected duality gap below $\epsilon$}
	
	\vspace{3pt}
	\parbox{1.1\textwidth}{
	$L\gets\nA$,~$\alpha \gets 
	L\sqrt{\frac{n+m}{\nnz(A)}}$,~$K\gets 
	\ceil{\frac{\log(2m)\alpha}{\epsilon}}$,~$\eta \gets 
	\frac{\alpha}{24 L^2}$,~$\tau\gets 
	\frac{1}{\eta}$,~$T\gets\ceil{\frac{4}{\eta\alpha}}$, 
	$(x_0, y_0) \gets (\boldsymbol{0}_n, 
	\tfrac{1}{m}\boldsymbol{1}_m)$
	} %

	\vspace{3pt}
	\For{$k=1,\ldots,K$}
	{
		\vspace{3pt}%
		\Comment{\emph{Relaxed oracle query:}}
		
		$\displaystyle (x_0, y_0) \gets (z_{k-1}\x,z_{k-1}\y)$, $(g\x_0, g\y_0) 
		\gets (A^\top y_0, -A x_0)$\;
		
		\vspace{3pt}
		
		\For{$t=1,\ldots,T$}
		{
			\vspace{3pt}%
			\Comment{\emph{Gradient estimation:}}
			\vspace{-4pt}%
			
			Sample $i\sim p$ where $\displaystyle p_i = \frac{\left| [y_{t-1}]_i 
				- [y_{0}]_i \right|}{\lone{y_{t-1}-y_0}}$, sample $j\sim q$ where 
			$\displaystyle q_j = \frac{ \left( [x_{t-1}]_j - [x_{0}]_j 
				\right)^2}{\ltwo{x_{t-1}-x_0}^2}$\;
			
			Set $\displaystyle\tilde{g}_{t-1} = g_0 +  \left(
			\ai \frac{[y_{t-1}]_i - [y_{0}]_i}{p_i}, 
			-\trunc\left(\aj \frac{[x_{t-1}]_j - [x_{0}]_j}{q_j}\right)\right)$
			\;
			\Comment*[f]{$[\trunc(v)]_k \defeq \min\{\tau, \max\{-\tau, 
			[v]_k\}\}$}

			\Comment{\emph{Mirror descent step:}}
			
			$\displaystyle x_t \gets \Pi_{\xset}\left( \frac{1}{1+\eta\alpha/2} 
			\left(  x_{t 
				- 1} + \frac{\eta\alpha}{2} x_0 - \eta \tilde{g}_{t-1}\x 
			\right)\right)$
			\Comment*[f]{$\Pi_{\xset}(v)=\frac{v}{\max\{1,\ltwo{v}\}}$}
			\;

			$\displaystyle y_t \gets \Pi_{\yset}\left( \frac{1}{1+\eta\alpha/2} 
			\left(
			\log y_{t - 1} + 
			\frac{\eta\alpha}{2}\log 
			y_0 - \eta \tilde{g}_{t-1}\y
			\right)\right)$
			\Comment*[f]{$\Pi_{\yset}(v)=\frac{e^v}{\lone{e^v}}$}
			\; 
		}
		$\displaystyle z_{k - 1/2} \gets \frac{1}{T}\sum_{t=1}^T (x_t, y_t)$\;

		\Comment{\emph{Extragradient step:}}
		
		$\displaystyle z_k\x \gets \Pi_{\xset}\left( z_{k-1}\x 
		-\tfrac{1}{\alpha} A^\top z_{k-1/2}\y \right)$, 
		$\displaystyle z_k\y \gets \Pi_{\yset}\left(\log z_{k-1}\y 
		+\tfrac{1}{\alpha} A z_{k-1/2}\x \right)$\;

	}
	\Return $\displaystyle \frac{1}{K}\sum_{k=1}^K z_{k-1/2}$
	\caption{Variance reduction for $\ell_2$-$\ell_1$ games}
	\label{alg:l1l2}
\end{algorithm}
}
\arxiv{\algTwoOne

With the improved gradient estimator and its analysis established, we  
combine it with our framework in Section~\ref{sec:framework} and obtain a 
complete variance reduction algorithm for $\ell_2$-$\ell_1$ games; 
Algorithm~\ref{alg:l1l2} is the result. It 
enjoys the following performance guarantee. 
}%
\notarxiv{
Algorithm~\ref{alg:l1l2} in~\Cref{app:l2l1-alg} combines our clipped 
gradient estimator with our general variance reduction framework. The 
analysis in~\Cref{app:cbb} gives the following guarantee.
}

\begin{theorem}
	\label{thm:l1l2}
	Let $A\in\R^{m\times n}$, $\epsilon > 0$, and any $\alpha \geq 
	\epsilon / \log(2m)$. 
	Algorithm~\ref{alg:l1l2} outputs a point 
	$z=(z\x,z\y)$ such that
	\arxiv{
	\begin{equation*}
	\E\, \left[\max_{y\in\Delta^m} y^\top A z\x - \min_{x\in\ball^n} 
	(z\y)^\top A x \right]
	=
	\E \left[\max_{i}\, {[Az\x]}_i + \ltwos{A^\top 
	z\y}\right] \le \epsilon,
	\end{equation*}
	}%
	\notarxiv{
	$\E\, \left[\max_{y\in\Delta^m} y^\top A z\x - \min_{x\in\ball^n} 
	(z\y)^\top A x \right]
	=
	\E \left[\max_{i}\, {[Az\x]}_i + \ltwos{A^\top 
		z\y}\right] \le \epsilon$,
	}
	and runs in time 
	\begin{equation}\label{eq:l1l2-time}
	O\left(\left(\nnz(A) + \frac{(m + n)\nA^2}{\alpha^2} 
	\right)\frac{\alpha\log(2m)}{\epsilon}\right).
	\end{equation}
	Setting $\alpha$ optimally, the running time is 
	\begin{equation}\label{eq:l1l2-time-opt}
	O\left(\nnz(A) + \frac{\sqrt{\nnz(A) (m + n)} 
		\nA\log(2m)}{\epsilon}\right).
	\end{equation}

\end{theorem}
\arxiv{
\begin{proof}
	The proof is identical to that of Theorem~\ref{thm:l1l1}, except 
	Proposition~\ref{prop:innerloop-cbb} replaces 
	Corollary~\ref{cor:innerloop-oracle}, $L$ is now $\nA$ instead of 
	$\norm{A}_{\max}$, and  $\Theta = \max_{z'}r(z')- \min_z r(z) = 
	\half+\log m \le \log(2m)$ rather than $\log(mn)$.
\end{proof}
}

\newcommand{\presentTwoTwo}{
\paragraph{Setup.} In the $\ell_2$-$\ell_2$ setup, both $\xset=\ball^n$ 
and 
$\yset=\ball^m$ 
are Euclidean unit balls, the norm over $\zset=\xset\times\yset$ is the 
Euclidean norm (which is dual to itself), and the distance generating 
function is $r(z)=\half\ltwo{z}^2$. Under the 
Euclidean norm, the Lipschitz constant of $g$ is $\norm{A}_{2\to2}$ (the 
largest singular value of $A$), and we also consider the Frobenius norm 
$\lfro{A}=(\sum_{i,j}A_{ij}^2)^{1/2}$, i.e.\ the Euclidean norm of the 
singular values of $A$. 

\begin{remark}
	In the $\ell_2$-$\ell_2$ setup, problems of the form 
	$\min_{x\in\ball^n}\max_{y\in\ball^m} y^\top A x$ are trivial, since the 
	saddle point is always the origin. However, as we explain in 
	Section~\ref{ssec:composite}, 
	our results extend to problems of the form 
	$\min_{x\in\ball^n}\max_{y\in\ball^m} \left\{y^\top A x + \phi(x) - 
	\psi(y)\right\}$ for convex functions $\phi,\psi$, e.g.\ 
	$\min_{x\in\ball^n}\max_{y\in\ball^m} \left\{y^\top A x + b^\top x +  
	c^\top y\right\}$, which are nontrivial.
\end{remark}

Our centered gradient estimator for the $\ell_2$-$\ell_2$ setup is of the 
form~\eqref{eq:tgdef}, where we sample from
\begin{equation}
\label{eq:l2l2-probs}
\begin{aligned}
p_i(w)=\frac{([w\y]_i-[w_0\y]_i)^2}{\left\Vert w\y - 
	w_0\y\right\Vert^2_{2}}~~\mbox{and}~~
\ q_j(w)=\frac{([w\x]_j-[w_0\x]_j)^2}{\left\Vert w\x - 
	w_0\x\right\Vert^2_{2}}.
\end{aligned}
\end{equation}
The resulting gradient estimator has the explicit form
\begin{equation}
\label{eq:tgdef-l2l2}
\tilde{g}_{w_0}(w) = g(w_0) + \left(
\ai \frac{\ltwos{w\y - w_0\y}^2}{[w\y-w_0\y]_i}, 
-\aj  \frac{\ltwos{w\x - w_0\x}^2}{[w\x-w_0\x]_j}\right).
\end{equation}

\begin{restatable}{lemma}{restateNaive}
	\label{lem:l2l2-gradient-est-dynamic}
	In the $\ell_2$-$\ell_2$ setup, the estimator~\eqref{eq:tgdef-l2l2} is 
	$(w_0,L)$-centered with $L = 
	\lfro{A}$.
\end{restatable}

\begin{proof}
	Unbiasedness follows from the estimator definition. The second property 
	follows from
	\begin{align*}
	\Ex{}\norm{\tilde{g}_{w_0}(w) - g(w_0)}_{2}^2
	 &= \sum\limits_{i \in [m]} 
	\frac{\norm{\ai}_2^2}{p_i}([w\y]_i-[w_0\y]_i)^2 + \sum\limits_{j\in[n]} 
	\frac{\norm{\aj}_2^2}{q_j}([w\x]_j-[w_0\x]_j)^2
	 \\&= \norm{A}_\mathrm{F}^2 \norm{w - w_0}_2^2.
	\end{align*}
\end{proof}

\notarxiv{
In~\Cref{ssec:l2l2-add} we provide two additional sampling distribution 
that yield estimators with the same guarantee. We may use these }%
\arxiv{We may use this }gradient estimator to build an algorithm with a 
convergence guarantee similar to Theorem~\ref{thm:l1l2}, except with 
$\lfro{A}$ instead of $\nA$ and $1$ instead of $\log(2m)$. This result 
improves the runtime of~\citet{BalamuruganB16} by a 
$\log({1}/{\epsilon})$ factor. However, as we discuss in 
Section~\ref{ssec:related}, unlike our $\ell_1$-$\ell_1$ and 
$\ell_2$-$\ell_1$ results, it is not a strict improvement over the 
linear-time mirror-prox method, which in the $\ell_2$-$\ell_2$  setting 
achieves running time $O( \norm{A}_{2\to2} \nnz(A) \epsilon^{-1})$. The 
regime in which our variance-reduced method has a stronger guarantee 
than mirror-prox is
\begin{equation*}
\mathrm{srank}(A) \defeq \frac{\lfro{A}^2}{\norm{A}_{2\to2}^2} \ll 
\frac{\nnz(A)}{n+m},
\end{equation*}
i.e.\ when the spectral sparsity of $A$ is significantly greater than its 
spatial sparsity.

We remark that $\ell_2$-$\ell_2$ games are closely related to linear 
regression, as
\begin{equation*}
\min_{x\in\ball^n} \ltwo{Ax-b}^2 = 
\left(
\min_{x\in\ball^n} 
\max_{y\in\ball^m} \big\{ y^\top A x - y^\top b \big\}
\right)^2.
\end{equation*}
The smoothness of the objective $\ltwo{Ax-b}^2$ is 
$\norm{A}_{2\to2}^2$, but runtimes of stochastic linear regression solvers 
typically depend on $\lfro{A}^2$ instead~\cite{StrohmerV09, Johnson13, 
Shalev-Shwartz13,  FrostigGKS15, LinMH15, Shalev-Shwartz16,
SchmidtRB17, Zhu17}. Viewed in this context, it is not surprising that our 
$\ell_2$-$\ell_2$ runtime scales as it does.
}
\arxiv{\subsection{$\ell_2$-$\ell_2$ games}\label{ssec:l2l2}}
\arxiv{\presentTwoTwo} %
\section{Extensions}\label{sec:ext}
In this section we collect a number of results that extend our framework 
and its applications. In \Cref{ssec:highprecision} we show how to use 
variance reduction to solve the proximal subproblem to high accuracy. This 
allows us to implement a relaxed gradient oracle for any monotone operator 
that admits an appropriate gradient estimator, overcoming a technical 
limitation in the analysis of Algorithm~\ref{alg:innerloop} (see discussion 
following Corollary~\ref{cor:innerloop-oracle}). In 
Section~\ref{ssec:composite}  we explain how to extend our results to 
composite saddle point problems of the form 
$\min_{x\in\xset}\max_{y\in\yset} \left\{f(x,y) + \phi(x) - \psi(y)\right\}$, 
where $f$ admits a centered gradient estimator and $\phi,\psi$ are convex 
functions. In~\Cref{ssec:strongly} we propose a small modification of 
Algorithm~\ref{alg:outerloop} that allows our framework to achieve linear 
convergence under strong-convexity-concavity of the objective (possibly 
arising from composite terms as in Section~\ref{ssec:composite}). In 
Section~\ref{ssec:finitesum}, we give a centered gradient estimator and 
resulting runtime bounds for finite-sum minimax problem of the form 
$\min_{x\in\xset}\max_{y\in\yset}\sum_{k \in [K]} f_k(x,y)$, when the 
gradient mapping for $f_k$ is Lipchitz continuous for every $k$. Finally, 
in Section~\ref{sec:additional} we return to the bilinear case and provide a 
number of alternative gradient estimators for the $\ell_2$-$\ell_1$ and 
$\ell_2$-$\ell_2$ settings. %
\subsection{High precision proximal mappings via variance reduction}
\label{ssec:highprecision}

\newcommand{\sind}[1]{^{(#1)}}
\newcommand{\zs}{z_\alpha}

Here we describe how to use gradient estimators that satisfy 
Definition~\ref{def:w-l-centered} to obtain high precision approximations 
to the exact proximal mapping, as well as a relaxed proximal oracle valid 
beyond the bilinear case. 
Algorithm~\ref{alg:highprecision} is a modification of 
Algorithm~\ref{alg:innerloop}, where we restart the mirror-descent iteration 
$N$ times, with each restarting constituting a \emph{phase}. In each phase, 
we re-center the gradient estimator $g$, but 
regularize towards the original initial point $w_0$. To analyze 
the performance of  the algorithm, we require two properties of proximal 
mappings with general Bregman divergences~\eqref{eq:prox-def}.

\begin{restatable}{lemma}{restateRegLB}\label{lem:reg-lb}
	Let $g$ by a monotone operator, let $z\in\zset$ and let $\alpha > 0$. 
	Then, for every $w\in\zset$, $\zs = \prox{z}{g}$ satisfies
	\begin{equation*}
	\inner{g(w)+\alpha\grad V_{z}(w)}{w-\zs} \ge \alpha V_{\zs}(w) + \alpha 
	V_{w}(\zs).
	\end{equation*}
\end{restatable}
\begin{proof}
	By definition of $\zs$, $\inner{g(\zs)+\alpha \grad V_{z}(\zs)}{\zs - w}\le 
	0$ for all $w\in\zset$. Therefore
	\begin{flalign*}
	\inner{g(w)+\alpha\grad V_{z}(w)}{w-\zs} & \ge 
	\inner{g(w)+\alpha\grad V_{z}(w)}{w-\zs} + \inner{g(\zs)+\alpha \grad 
	V_{z}(\zs)}{\zs - w}
	\\ &=
	\inner{g(w)-g(\zs)}{w-\zs}  + \alpha \inner{\grad V_{z}(w)-\grad 
	V_{z}(\zs)}{w-\zs} 
	\\ & \underset{(i)}{\ge} 
	\alpha \inner{\grad V_{z}(w)-\grad V_{z}(\zs)}{w-\zs} 
	\underset{(ii)}{=}   \alpha V_{\zs}(w) + \alpha V_{w}(\zs),
	\end{flalign*}
	where $(i)$ follows from monotonicity of $g$ and $(ii)$ holds by 
	definition of the Bregman divergence.
\end{proof}

\begin{restatable}{lemma}{restateRegLB}\label{lem:smooth-breg-ub}
	Let $g$ be a monotone operator and let $\alpha>0$. Then, for every 
	$z\in\zset$, $\zs = \prox{z}{g}$ satisfies
	\begin{equation*}
	V_{\zs}(z)  + V_{z}(\zs) \le \frac{\norm{g(z)}_*\norm{z-\zs}}{\alpha} \le 
	\frac{\norm{g(z)}_*^2}{\alpha^2}.
	\end{equation*}
\end{restatable}
\begin{proof}
	Using Lemma~\ref{lem:reg-lb} with $w=z$ gives
	\begin{equation*}
	\alpha V_{\zs}(z) + \alpha 
	V_{z}(\zs) \le \inner{g(z)+\alpha\grad V_{z}(z)}{z-\zs} \le \inner{g(z)}{z-\zs},
	\end{equation*}
	where we used the fact that $z$ minimizes the convex function 
	$V_z(\cdot)$ and 
	therefore $\inner{\grad V_z(z)}{z-u}\le 0$ for all $u\in\zset$. 
	Writing $\inner{g(z)}{z-\zs} \le \norm{g(z)}_*\norm{z-\zs}$ gives the first 
	bound in the lemma. Next, strong convexity of $r$ implies
	\begin{equation*}
	\norm{z-\zs}^2 \le V_{\zs}(z) + V_{z}(\zs) \le \frac{ 
	\norm{g(z)}_*\norm{z-\zs}}{\alpha},
	\end{equation*}
	and the second bound follows from dividing by $\norm{z-\zs}$.
\end{proof}

\begin{algorithm}
	\label{alg:highprecision}
	\DontPrintSemicolon
	\KwInput{Initial $w_0\in\zset$, centered gradient estimator 
	$\tilde{g}_{z}$ $\forall z\in\zset$, 
		oracle quality $\alpha>0$}
	\Parameter{Step size $\eta$, inner iteration count $T$, phase count	$N$}
	\KwOutput{Point $\hat{w}_N$ satisfying $\E 
	V_{\hat{w}_N}(\zs) \le 2^{-N} 
	V_{w_0}(\zs)$ where $\zs = 
	\prox{w_0}{g}$
		(for 
		appropriate $\tilde{g}$, $\eta$, $T$)}
	Set $\hat{w}_0 \gets w_0$\;
	\For{$n = 1, \ldots, N$}
	{
	Prepare centered gradient estimator $\tilde{g}_{\hat{w}_{n-1}}$
	\Comment*[f]{e.g., by 
	computing $g(\hat{w}_{n - 1})$}\; 
	Draw $\hat{T}$ uniformly from $[T]$\;
	$w_0\sind{n}\gets \hat{w}_{n-1}$\;
	\For{$t = 1, \ldots, \hat{T}$}
	{
		$w_t\sind{n} \leftarrow 
		\argmin_{w\in\zset}\left\{\inners{
			\tilde{g}_{\hat{w}_{n-1}}\big(w_{t-1}\sind{n}\big)}{w} + 
			\alpha V_{w_0}(w) + \frac{1}{\eta}V_{w_{t 
				- 1}\sind{n}}(w) 
		\right\}$\; \label{line:restarted-step}
	}
	$\hat{w}_n\gets w_{\hat{T}}\sind{n}$\;
	}
	\Return $\hat{w}_N$
	\caption{$\RestartedInnerLoop(w_0, z\mapsto \tilde{g}_{z}, \alpha)$}
\end{algorithm}

We now state the main convergence result for 
Algorithm~\ref{alg:highprecision}.

\begin{restatable}{proposition}{restateRestartedLoop}\label{prop:restartedloop}
	Let $\alpha,L>0$, let $w_0\in\zset$, let $\tilde{g}_{z}$ be 
	$(z,L)$-centered for monotone $g$ and every $z\in\zset$ and let $\zs = 
	\prox{w_0}{g}$. Then, for $\eta = 
	\frac{\alpha}{8L^2}$,  
	$T \ge  
	\frac{4}{\eta\alpha}=\frac{32L^2}{\alpha^2}$, and any $N\in\N$ the 
	output $\hat{w}_N$ of  
	Algorithm~\ref{alg:highprecision} satisfies
	\begin{equation}\label{eq:restartedloop-guarantee}
	\Ex{} V_{\hat{w}_N}(\zs) \le 2^{-N} V_{w_0}(\zs).
	\end{equation}
\end{restatable}
\begin{proof}
	Fix a phase $n\in[N]$. For every $u\in\zset$ we have the mirror descent 
	regret bound
	\begin{equation*}
	\sum_{t\in[T]} \inner{\tilde{g}_{\hat{w}_{n-1}}\big(w_t\sind{n}\big) + 
	\alpha \grad 
		V_{w_0}\big(w_t\sind{n}\big)}{w_t\sind{n}-u} \le 
	\frac{V_{\hat{w}_{n-1}}(u)}{\eta} 
	+ \frac{\eta}{2}\sum_{t\in[T]} 
	\norm{
		\tilde{g}_{\hat{w}_{n-1}}\big(w_t\sind{n}\big)-g(\hat{w}_{n-1})}_*^2;
	\end{equation*}
	see Lemma~\ref{lem:mirror-descent} in 
	Appendix~\ref{app:mirror-descent}, with $Q(z)=\eta 
	\inner{g(\hat{w}_{n-1})}{z}+\eta\alpha 
	V_{w_0}(z)$. Choosing $u=\zs$, taking expectation and using 
	Definition~\ref{def:w-l-centered} gives
	\begin{equation}\label{eq:restarts-regret}
	\E\sum_{t\in[T]} \inner{g\big(w_t\sind{n}\big) + 
		\alpha \grad 
		V_{w_0}\big(w_t\sind{n}\big)}{w_t\sind{n}-\zs} \le 
	\frac{\E V_{\hat{w}_{n-1}}(\zs)}{\eta} 
	+ \frac{\eta L^2}{2}\sum_{t\in[T]} 
	\E\norm{
		w_t\sind{n}-\hat{w}_{n-1}}^2.
	\end{equation}
	(Note that $\zs$ is a function of $w_0$ and hence independent of 
	stochastic gradient estimates.) By the triangle inequality and strong 
	convexity of $r$,
	\begin{equation}\label{eq:restarts-rhs-bound}
	\norms{w_t\sind{n}-\hat{w}_{n-1}}^2
	\le 
	2\norms{\zs-\hat{w}_{n-1}}^2 + 2\norms{w_t\sind{n}-\zs}^2
	\le 4 V_{\hat{w}_{n-1}}(\zs) + 4 V_{\zs}\big(w_t\sind{n}\big). 
	\end{equation}
	By Lemma~\ref{lem:reg-lb} we have that for every $t\in [T]$
	\begin{equation}\label{eq:restarts-lhs-bound}
	\inner{g\big(w_t\sind{n}\big) + 
		\alpha \grad 
		V_{w_0}\big(w_t\sind{n}\big)}{w_t\sind{n}-\zs} 
	\ge \alpha V_{w_t\sind{n}}(\zs) + \alpha V_{\zs}\big(w_t\sind{n}\big).
	\end{equation}
	Substituting the bounds~\eqref{eq:restarts-rhs-bound} 
	and~\eqref{eq:restarts-lhs-bound} into the expected regret 
	bound~\eqref{eq:restarts-regret} and rearranging gives
	\begin{equation*}
	\frac{1}{T}\sum_{t\in[T]} \E V_{w_t\sind{n}}(\zs) 
	\le 
	\left(\frac{1}{\eta \alpha T} + \frac{2\eta L^2}{\alpha} \right) \E 
	V_{\hat{w}_{n-1}}(\zs)
	+ \frac{2\eta L^2 - \alpha}{\alpha T}\sum_{t\in[T]} 
	\E  V_{\zs}\big(w_t\sind{n}\big) \le \frac{1}{2}\, \E V_{w_t\sind{n-1}}(\zs),
	\end{equation*}
	where in the last transition we substituted $\eta = \frac{\alpha}{8L^2}$ 
	and 
	$T \ge \frac{4}{\eta \alpha}$. Noting that $\frac{1}{T}\sum_{t\in[T]} 
	\E V_{w_t\sind{n}}(\zs)  = \E V_{\hat{w}_n}(\zs)$ and recursing on $n$ 
 	completes the proof.
\end{proof}

The linear convergence bound~\eqref{eq:restartedloop-guarantee} 
combined with Lemma~\ref{lem:smooth-breg-ub} implies that 
Algorithm~\ref{alg:highprecision} implements a relaxed proximal oracle.
\begin{restatable}{corollary}{restateRestartedLoopOracle}
	\label{cor:restartedloop-oracle}
	Let $G,D>0$ be such that $\norm{g(z)}_* \le G$ and 
	$\norm{z-z'}\le D$ for every $z,z'\in\zset$ and let $\veps > 0$. 
	Then, in 
	the setting of Proposition~\ref{prop:restartedloop} with $N \ge 
	1+2 \log_2\left(\frac{G(G+2LD)}{\alpha \veps}\right)$, we have that 
	$\mc{O}(w_0) = 
	\RestartedInnerLoop(w_0, \tilde{g}, \alpha)$ is an 
	$(\alpha,\veps)$-relaxed proximal oracle.
\end{restatable}
\begin{proof}
	Let $\hat{w}=\RestartedInnerLoop(w_0, \tilde{g}, \alpha)$ and let $\zs = 
	\prox{w_0}{g}$. For every $u\in\zset$, we have
	\begin{equation*}
	\inner{g(\hat{w})}{\hat{w}-u}
	=
	\inner{g(\zs)}{\zs-u}
	+
	\inner{g(\zs)}{\hat{w}-\zs}
	+
	\inner{g(\hat{w})-g(\zs)}{\hat{w}-u}.
	\end{equation*}
	By the definition~\eqref{eq:prox-def} of $\zs$ we have 
	$\inner{g(\zs)}{\zs-u} \le \alpha V_{w_0}(u)$. By Hölder's inequality and 
	the assumption that $g$ is bounded, we have  
	$\inner{g(\zs)}{\hat{w}-\zs} \le G \norm{\hat{w}-\zs}$. Finally, since $g$ 
	is $2L$-Lipschitz (see Remark~\ref{rem:lipschitz}) and 
	$\norm{\hat{w}-u}\le D$ by assumption, we have 
	$\inner{g(\hat{w})-g(\zs)}{\hat{w}-u} \le 2LD\norm{\hat{w}-\zs}$. 
	Substituting back these three bounds and rearranging yields
	\begin{equation*}
	\inner{g(\hat{w})}{\hat{w}-u} - \alpha V_{w_0}(u)
	\le 
	(G + 2LD)\norm{\hat{w}-\zs}
	\le 
	(G + 2LD)\sqrt{2V_{\hat{w}}(\zs)},
	\end{equation*}
	where the last bound is due to strong convexity of $r$. Maximizing over 
	$u$ and taking expectation, we have by Jensen's inequality and  
	Proposition~\ref{prop:restartedloop},
	\begin{equation*}
	\E \max_{u\in\zset} \left\{\inner{g(\hat{w})}{\hat{w}-u} - \alpha 
	V_{w_0}(u)\right\}
	\le 
	(G + 2LD)\sqrt{2\E V_{\hat{w}}(\zs)}
	\le 
	2^{-(N-1)/2}(G + 2LD)\sqrt{V_{w_0}(\zs)}.
	\end{equation*}
	Lemma~\ref{lem:smooth-breg-ub} gives us $\sqrt{V_{w_0}(\zs)} \le 
	\sqrt{\norm{g(w_0)}_*^2 / \alpha ^2} \le G/\alpha$, and therefore 
	$N\ge 1+
	2 \log_2\left(\frac{G(G+2LD)}{\alpha \veps}\right)$ establishes the 
	oracle property $\E \max_{u\in\zset} \left\{\inner{g(\hat{w})}{\hat{w}-u} 
	- \alpha 
	V_{w_0}(u)\right\} \le \veps$.
\end{proof}

\begin{remark}\label{rem:high-precision-extension}
	In the $\ell_2$-$\ell_1$ setup of Section~\ref{ssec:l1l2}, 
	Proposition~\ref{prop:restartedloop} and 
	Corollary~\ref{cor:restartedloop-oracle} extend straightforwardly to 
	centered-bounded-biased gradient estimators 
	(Definition~\ref{def:z-l-cbb}) using arguments from the proof 
	of Proposition~\ref{prop:innerloop-cbb}. 
\end{remark}

Since Algorithm~\ref{alg:highprecision} computes a highly accurate 
approximation of the proximal mapping, it is reasonable to expect that 
directly iterating $z_k = \RestartedInnerLoop(z_{k-1}, \tilde{g}, \alpha)$ for 
$k\in [K]$
would yield an $O(\alpha \Theta / K)$ error bound, without requiring the 
extragradient step in Algorithm~\ref{alg:outerloop}. However, we could not 
show such a bound without additionally requiring uniform smoothness 
of the distance generating function $r$, which does not hold for the 
negative entropy we use in the $\ell_1$ setting.

\subsection{Composite problems}\label{ssec:composite}

Consider the ``composite'' saddle point problem of the form
\begin{equation*}
\min_{x\in\xset}\max_{y\in\yset} 
\left\{f(x,y) + \phi(x) - \psi(y)\right\},
\end{equation*}
where $\grad f$ admits a centered gradient estimator and $\phi,\psi$ are 
``simple'' convex functions in the sense they have efficiently-computable 
proximal mappings. As usual in convex optimization, it is straightforward 
to extend our framework to this setting. Let $\Upsilon(z) \defeq \phi(z\x) + 
\psi(z\y)$ so that $g(z)+\grad\Upsilon(z)$ denotes a (sub-)gradient 
mapping for the composite problem at point $z$. Algorithmically, the 
extension consists of changing Line~\ref{line:outer-extragrad} of
Algorithm~\ref{alg:outerloop} to
\begin{equation*}
z_k \leftarrow
\argmin_{z \in \zset}\left\{ 
\inner{g\left(z_{k-1/2}\right)+\grad\Upsilon(z_{k-1/2})}{z} 
+
\alpha V_{z_{k-1}}(z)\right\},
\end{equation*}
changing line~\ref{line:inner-step} of Algorithm~\ref{alg:innerloop} to
\begin{equation*}
w_t \leftarrow \argmin_{w\in\zset}\left\{\inner{\tilde{g}_{w_0}(w_{t - 
		1})}{w} + \Upsilon(w) + \frac{\alpha}{2}V_{w_0}(w) + 
\frac{1}{\eta}V_{w_{t - 1}}(w)
\right\},
\end{equation*}
and similarly adding
 $\Upsilon(w)$ 
 to the minimization in 
line~\ref{line:restarted-step} of Algorithm~\ref{alg:highprecision}. 

Analytically, we replace $g$ with $g+\grad\Upsilon$ %
 in the duality gap bound~\eqref{eq:gap-bound}, 
Definition~\ref{def:alphaprox} (relaxed proximal oracle), and 
Proposition~\ref{prop:outerloopproof} and its proof, which holds without 
further change. To implement the composite relaxed proximal oracle we 
still assume a centered gradient estimator for $g$ only. However, with the 
algorithmic modifications described above, the 
guarantee~\eqref{eq:innerloop-guarantee} of  
Proposition~\ref{prop:innerloopproof} now has $g+\grad\Upsilon$ instead 
of $g$; the 
only change to the proof is that we now invoke 
Lemma~\ref{lem:mirror-descent} (in Appendix~\ref{app:mirror-descent}) 
with the 
composite term $\eta \left[\inner{g(w_0)}{z} + \Upsilon(z) + 
\tfrac{\alpha}{2}V_{w_0}(z)\right]$, and the 
bound~\arxiv{\eqref{eq:composite-regret-basic}}%
\notarxiv{\eqref{eq:composite-regret-basic-rep}} becomes
\begin{equation*}
\sum_{t\in[T]} \inner{\tilde{g}_{w_0}(w_t) + 
\grad\Upsilon(w_t)+\tfrac{\alpha}{2}\grad 
	V_{w_0}(w_t)}{w_t-u} \le \frac{V_{w_0}(u)}{\eta} 
+ \frac{\eta}{2}\sum_{t\in[T]} \norms{\tilde{\delta}_t}_*^2.
\end{equation*}
Proposition~\ref{prop:innerloop-cbb}, Proposition~\ref{prop:restartedloop} 
and Corollary~\ref{cor:restartedloop-oracle} similarly extend to the 
composite setup.

The only point in our development that does not immediately extend to the 
composite setting is Corollary~\ref{cor:innerloop-oracle} and its 
subsequent discussion. There, we argue that Algorithm~\ref{alg:innerloop} 
implements a relaxed proximal oracle only when $\inner{g(z)}{z-u}$ is 
convex in $z$ for all $u$, which is the case for bilinear $f$. However, 
this condition might fail for $g+\grad \Upsilon$ even when it holds for 
$g$. In this case, we may still use the oracle implementation 
guaranteed by Corollary~\ref{cor:restartedloop-oracle} for any convex 
$\Upsilon$. 

\subsection{Strongly monotone problems}\label{ssec:strongly}

In this section we consider
variational inequalities with \emph{strongly monotone} operators. 
Following~\citet{LuFN18}, we say that an 
operator $g$ is $\mu$-strongly monotone relative to the distance 
generating function $r$ if it satisfies
\begin{equation}
\label{eq:sm-def}
\inner{g(z')-g(z)}{z'-z}\ge \mu  \cdot V_{z'}(z)~~\mbox{for all}~~ 
z',z\in\zset.	
\end{equation}
By strong convexity of $r$, a $\mu$-strongly monotone operator relative 
to $r$ is also $\mu$-strongly monotone in the standard sense, i.e.,  
$\inner{g(z')-g(z)}{z'-z} \ge 
\frac{\mu}{2} \norm{z'-z}^2$ for all $z,z'\in\zset$. For saddle point 
problems $\min_{x\in\xset}\max_{y\in\yset}f(x,y)$ with gradient mapping  
$g(x,y)=(\grad_x f(x,y), -\grad_y f(x,y))$ and separable 
$r(x,y)=r\x(x)+r\y(y)$, the 
property~\eqref{eq:sm-def} holds whenever $f(x,y)- \mu r\x(x) + \mu 
r\y(y)$ is convex-concave, i.e., whenever $f(\cdot, y)$ is  
$\mu$-strongly convex relative to $r\x$ for every $y$ and $f(x,\cdot)$ 
is $\mu$-strongly-concave relative to $r\y$ for every 
$x$.

Strongly monotone operators have a unique saddle point $z^\star$ 
satisfying  
\begin{equation}
\label{eq:gen-opt}
\max_{u\in\zset}\inner{g(z^\star)}{z^\star-u}\le 0.
\end{equation}
We show that Algorithm~\ref{alg:outerloop-sm}---a small modification of 
Algorithm~\ref{alg:outerloop}---converges linearly to $z^\star$ (i.e., with 
the 
distance to $z^\star$ decaying 
exponentially in the number of iterations).
The main difference between the algorithms is that the extra-gradient step 
in Algorithm~\ref{alg:outerloop-sm} has an additional regularization term 
around $z_{k-1/2}$, making it similar in form to the regularized  
mirror descent steps we use in Algorithm~\ref{alg:innerloop}. In addition, 
Algorithm~\ref{alg:outerloop-sm} outputs the last iterate rather the 
iterate average.

\begin{algorithm}
	\label{alg:outerloop-sm}
	\DontPrintSemicolon
	\KwInput{$(\alpha,\veps)$-relaxed proximal oracle $\mathcal{O}(z)$ for 
	gradient mapping $g$ satisfying~\eqref{eq:sm-def}}
	\Parameter{Number of iterations $K$}
	\KwOutput{Point $z_K$ with $\E V_{z_K}(z^\star)\le 
	(\tfrac{\alpha}{\mu+\alpha})^K\Theta+\frac{\veps}{\mu}$}
	$z_0 \gets \argmin_{z\in\zset} r(z)$ \;
	\For{$k = 1, \ldots, K$}
	{
		$z_{k-1/2} \leftarrow \mathcal{O}(z_{k - 1})$ \Comment*[f]{We 
		implement $\mathcal{O}(z_{k-1})$ by calling $\InnerLoop(z_{k-1}, 
		\tilde{g}_{z_{k - 1}}, \alpha)$} \; 
		$z_k \leftarrow \argmin_{z \in \zset}\left\{ 
		\inner{g\left(z_{k-1/2}\right)}{z} + 
		\alpha V_{z_{k-1}}(z)+\mu V_{z_{k-1/2}}(z)\right\}$\; 
		\label{line:outer-extragrad-sm}
	}
	\Return $z_K$
	\caption{$\OuterLoopGen(\mathcal{O})$}
\end{algorithm}

\begin{proposition}%
\label{prop:outerloopproof-sm}
Let $\mathcal{O}$ be an ($\alpha$,$\veps$)-relaxed proximal oracle for a 
gradient mapping $g$ that is $\mu$-strongly monotone relative to  
distance-generating function $r$ with 
range at most $\Theta$. 
Let $z_K$ be the output of Algorithm~\ref{alg:outerloop-sm}. Then
\begin{equation*}
\E\, 
V_{z_K}(z^\star)\le\left(\frac{\alpha}{\mu+\alpha}\right)^K\Theta
+\frac{\veps}{\mu}.
\end{equation*}
If in addition $g$ is a gradient mapping for $f$ and $\norm{(\grad_x f(z), 
-\grad_y f(z'))}_*\le G$ for all $z,z'\in\zset$, then
\begin{equation*}
\E\, \gap(z_K) \le 
\sqrt{2}G\sqrt{\left(\frac{\alpha}{\mu+\alpha}\right)^K \Theta 
+\frac{\veps}{\mu}}.
\end{equation*}
\end{proposition}
\begin{proof}
Fix an iteration $k$. Using $\mu$-strong-monotonicity~\eqref{eq:sm-def} 
and optimality of $z^\star$~\eqref{eq:gen-opt} yields
\begin{flalign}
\mu V_{z_{k-1/2}}(z^\star) & \stackrel{\eqref{eq:sm-def}}{\le}  
\inner{g(z_{k-1/2})-g(z^\star)}{z_{k-1/2} - z^\star}
 \stackrel{\eqref{eq:gen-opt}}{\le}  \inner{g(z_{k-1/2})}{z_{k-1/2} - z^\star}
 \nonumber \\ &
 =  \inner{g(z_{k-1/2})}{z_{k} - z^\star} + \inner{g(z_{k-1/2})}{z_{k-1/2} - 
 z_{k}}.\label{eq:strong-eq-1}
\end{flalign}
Next, %
\begin{flalign}
	\inner{g(z_{k-1/2})}{z_{k} - z^\star} &\stackrel{(i)}{\leq}  -\langle\alpha 
	\nabla V_{z_{k-1}}(z_k)+\mu\nabla V_{z_{k-1/2}}(z_k),z_k-z^\star\rangle
	\nonumber \\ &
	 \stackrel{(ii)}{\leq} 
	 \mu\Big(V_{z_{k - 1/2}}(z^\star) - V_{z_k}(z^\star) - V_{z_{k - 
	 1/2}}(z_k)\Big)
+\alpha\Big(V_{z_{k-1}}(z^\star)-V_{z_{k}}(z^\star)-V_{z_{k-1}}(z_k)\Big)
\nonumber \\ &
\le \alpha V_{z_{k-1}}(z^\star)-(\mu+\alpha)V_{z_k}(z^\star)+\mu 
V_{z_{k-1/2}}(z^\star)-\alpha V_{z_{k-1}}(z_k),\label{eq:strong-eq-2}
\end{flalign}
where we use $(i)$ the first-order optimality condition for $z_k$ and $(ii)$ 
the three-point property of Bregman divergence~\eqref{eq:three-point}. In 
the last inequality, we also use nonnegativity of $V_{z_{k-1/2}}(z_k)$. Now, 
combining~\eqref{eq:strong-eq-1} and~\eqref{eq:strong-eq-2} yields
\begin{equation*}
\label{eq:strong-eq-both}	
\begin{aligned}
\mu V_{z_{k-1/2}}(z^\star) & \le \alpha 
V_{z_{k-1}}(z^\star)-(\mu+\alpha)V_{z_k}(z^\star)+\mu 
V_{z_{k-1/2}}(z^\star) +[\inner{g(z_{k-1/2})}{z_{k-1/2} - z_{k}}-\alpha 
V_{z_{k-1}}(z_k)].
\end{aligned}
\end{equation*}

Rearranging terms, taking an expectation, and using the  
Definition~\ref{def:alphaprox} of an $(\alpha,\veps)$-relaxed proximal 
oracle yields
\[
\E V_{z_k}(z^\star)\le\frac{\alpha}{\mu+\alpha}\E 
V_{z_{k-1}}(z^\star)+\frac{\veps}{\mu+\alpha}.
\]
Applying this bound recursively $K$ times and using that 
$V_{z_0}(u) = 
r(u)-r(z_0) \le \Theta$ for $z_0$ the minimizer of $r$, we have
\[
\E V_{z_K}(z^\star)\le\left(\frac{\alpha}{\mu+\alpha}\right)^K\Theta 
+\sum_{k = 0}^{K-1}\left(\frac{\alpha}{\mu + \alpha}\right)^{k}  
\left(\frac{\veps}{\mu + \alpha}\right) 
\le 
\left(\frac{\alpha}{\mu+\alpha}\right)^K\Theta +
\frac{\veps}{\mu}.
\]

To bound $\gap(z_K)$ (defined in~\eqref{eq:gap-def}), write 
$\gapu(z;u)=f(z\x,u\y)-f(u\x,z\y)$. Then we have 
$\norm{\grad_z\gapu(z;u)}_* =\norm{(\grad_x f(z\x,u\y),-\grad_y 
f(u\y,z\x))}_*\le G$ by assumption (since $\zset=\xset\times \yset$), and 
therefore $\gapu(z;u)$ is $G$-Lipschitz in $z$. Consequently, for any 
$u\in\zset$, \begin{equation*} \gapu(z_K;u) \le \gapu(z^\star;u) + 
G\norm{z_K - z^\star} \le G\norm{z_K - z^\star}, \end{equation*} where 
the second transition follows from the optimality of $z^\star$ (see 
Appendix~\ref{app:gap-bound}). Therefore, the 
definition~\eqref{eq:gap-def} of $\gap$, strong convexity of $r$ and 
Jensen's inequality yield \begin{equation*} \E \gap(z_K) = 
\E\max_{u\in\zset}  \gapu(z_K;u) \le G\, \E \norm{z_K - z^\star} \le \E G 
\sqrt{ 2V_{z_K}(z^\star)} \le \sqrt{2}G \sqrt{ \E V_{z_K}(z^\star)}. 
\end{equation*} Substituting the bound on $\E V_{z_K}(z^\star)$ 
concludes the proof.
\end{proof}

\begin{remark}
Algorithm~\ref{alg:outerloop-sm} and 
Proposition~\ref{prop:outerloopproof-sm} extend straightforwardly to the 
composite setting of Section~\ref{ssec:composite}. Using the notation of 
Section~\ref{ssec:composite}, we simply replace $g$ with $g+\grad 
\Upsilon$ in the condition~\eqref{eq:sm-def}, 
line~\ref{line:outer-extragrad-sm} of Algorithm~\ref{alg:outerloop-sm}, 
and Proposition~\ref{prop:outerloopproof-sm}. For the bound on $\E 
\gap(z_K)$ in 
Proposition~\ref{prop:outerloopproof-sm}, we require that   
$\norm{(\grad_x f(z)+\grad \phi(z\x), 
	-\grad_y f(z')+\grad \psi({z'}\y))}_*\le G$ for all $z,z'\in\zset$. (We also 
	still modify Algorithm~\ref{alg:innerloop} according to 
	Section~\ref{ssec:composite}.)
\end{remark}

Notably, Algorithm~\ref{alg:outerloop-sm} and 
Proposition~\ref{prop:outerloopproof-sm} work with the same relaxed 
proximal oracles as the rest of our framework, making our implementations 
via Algorithm~\ref{alg:innerloop} (or Algorithm~\ref{alg:highprecision}) 
and 
centered gradient estimators immediately applicable.
However the gradient mappings of bilinear functions are not 
strongly monotone. Therefore, to leverage strong monotonicity we 
consider the  
composite problem
\begin{align}
\min_{x\in \xset}\max_{y\in\yset} y^\top Ax+\phi(x)-\psi(y)
\label{eq:composite-bilinear}
\end{align}
where we assume that $\phi$ and $\psi$ admit simple proximal mappings, 
are Lipschitz,  and are strongly convex relative to the distance generating 
functions $r\x$ and $r\y$. More specifically, we assume that
\begin{equation}\label{eq:composite-bounded}
\norm{(\grad \phi(x), \grad \psi(y))}_* \le B~~\mbox{for all 
}x\in\xset~\mbox{and}~y\in\yset
\end{equation}
and
\begin{equation}\label{eq:composite-stronly}
\phi - \mu\x r\x~\mbox{and}~\psi - \mu\y r\y~\mbox{are convex},
\end{equation}
for some constants $B,\mu\x,\mu\y >0$.  

\newcommand{\rhat}{\hat{r}}
\newcommand{\Vhat}{\what{V}}

To handle the possibility that $\mu\x \ne \mu\y$ we use a norm-rescaling 
approach similar to the ``mixed setup'' argument of~\citet{Nemirovski04}. 
Let
\begin{equation*}
\rho \defeq \sqrt{\frac{\mu\x}{\mu\y}}
\end{equation*}
and define new distance generating functions
\begin{equation}\label{eq:rescaled-r}
\rhat\x(x) \defeq \rho r\x(x)
~~\mbox{and}~~\rhat\y(y) \defeq \frac{1}{\rho} r\x(y).
\end{equation}
By the assumption~\eqref{eq:composite-stronly} we have that $\phi - 
\sqrt{\mu\x\mu\y} \rhat\x$ and $\psi - \sqrt{\mu\x\mu\y} \rhat\y$ are 
convex, and therefore the objective~\eqref{eq:composite-bilinear} is 
$\sqrt{\mu\x\mu\y}$-strongly monotone relative to $\rhat(z) = 
\rhat\x(z\x)+\rhat\y(z\y)$. Using $\rhat$ instead of $r$ in 
Algorithms~\ref{alg:outerloop-sm} and~\ref{alg:innerloop} is equivalent to 
scaling $\eta, \alpha$ and $\mu$ separately for each of the $\xset$ and 
$\yset$ blocks; for ease of reference we give the complete resulting 
procedure as Algorithm~\ref{alg:full-loop}.

\begin{algorithm}
	\label{alg:full-loop}
	\DontPrintSemicolon
	\KwInput{A problem of the form~\eqref{eq:composite-bilinear} and a 
		centered gradient estimator $\tilde{g}$.}
	\Parameter{Strong-monotonicity $\mu=\sqrt{\mu\x\mu\y}$, step size 
		$\eta$, step size ratio $\rho = \sqrt{\mu\x / \mu\y}$, oracle 
		quality 
		$\alpha >0$, number of inner-loop iterations $T$, and number of 
		outer-loop iterations $K$.}
	\KwOutput{Point $z_K$ with $\E \left[\rho  
	V\x_{z_K}(z^\star)+\tfrac{1}{\rho } V\y_{z_K}(z^\star)\right]\le 
		(\tfrac{\alpha}{\mu+\alpha})^K(\rho +\tfrac{1}{\rho })\Theta$}
	$z_0 \gets \argmin_{z\in\zset} r(z)$ \;
	\For{$k = 1, \ldots, K$}
	{
		$w_0\gets z_{k-1}$\;
		\For{$t = 1, \ldots, T$}
		{
			$w_t\x \leftarrow 
			\argmin_{w\x\in\xset}\left\{\inner{\tilde{g}\x_{w_0}(w_{t - 
					1})}{w\x} + \phi(w\x) + \frac{\alpha}{2}\rho V\x_{w_0}(w) + 
			\frac{1}{\eta}\rho V\x_{w_{t - 1}}(w) 
			\right\}$\; 
			$w_t\y \leftarrow  
			\argmin_{w\y\in\yset}\left\{\inner{\tilde{g}\x_{w_0}(w_{t - 
					1})}{w\y} + \psi(w\y) + 
					\frac{\alpha}{2}\tfrac{1}{\rho }V\y_{w_0}(w) 
			+ 
			\frac{1}{\eta}\tfrac{1}{\rho }V\y_{w_{t - 1}}(w) 
			\right\}$\; 
		}
		$z_{k-1/2}=\frac{1}{T}\sum_{t=1}^T w_t$\;
		$z_k\x \leftarrow  \argmin_{z\x \in \xset}\left\{ 
		\inners{g\x\left(z_{k-1/2}\right) + \grad \phi(z_{k-1/2}\x)}{z\x} + 
		\alpha\rho  V\x_{z_{k-1}}(z)+\mu \rho  
		V\x_{z_{k-1/2}}(z)\right\}$\; 
		$z_k\y \leftarrow \argmin_{z\y \in \yset}\left\{ 
		\inners{g\y\left(z_{k-1/2}\right)+ \grad \psi(z_{k-1/2}\x)}{z\x} + 
		\alpha\frac{1}{\rho } V\y_{z_{k-1}}(z)+\mu \frac{1}{\rho } 
		V\y_{z_{k-1/2}}(z)\right\}$\;
		\label{line:outer-extragrad-sm}
	}
	\Return $z_K$
	\caption{$\FullLoopsm()$}
\end{algorithm}

To analyze the algorithm, we first note that for all of the setups in 
Section~\ref{sec:apps}, the norm on 
$\zset=\xset\times\yset$ satisfies $\norm{z}^2 = \norm{z\x}^2 + 
\norm{z\y}^2$ and that $r\x$ and $r\y$ are 1-strongly convex w.r.t.\ the 
corresponding block norms. Therefore, $\rhat$ is 1-strongly convex 
w.r.t.\ the rescaled norm
\begin{equation*}
\norm{z}_{\mathrm{r}}^2 \defeq \rho \norm{z\x}^2 + \frac{1}{\rho} 
\norm{z\y}^2.
\end{equation*}
Moreover, it is straightforward to verify that our $\ell_1$-$\ell_1$ and 
$\ell_2$-$\ell_2$ centered gradient estimators are also centered w.r.t.\ 
$\norm{\cdot}_\mathrm{r}$ and its dual $\norm{\gamma}_{\mathrm{r},*}^2 
= 
\frac{1}{\rho } \norm{\gamma\x}_*^2 + \rho  \norm{\gamma\y}_*^2$.  
For example, the proof of Lemma~\ref{lem:l1-gradient-est} changes to
\begin{equation*}
\norm{\tilde{g}_{w_0}(w) - g(w_0)}_{\mathrm{r},*}^2 
=  
\frac{1}{\rho }\linf{\ai}^2 \lones{w\y -w_0\y}^2 + 
\rho  \linf{\aj}^2 \lones{w\x -w_0\x}^2
\le \norm{A}_{\max}^2\norm{w - w_0}_{\mathrm{r}}^2.
\end{equation*}
The $\ell_2$-$\ell_1$ setup is less immediate, but retracing the proof of 
Proposition~\ref{prop:innerloop-cbb} shows that the 
guarantee~\eqref{eq:innerloop-guarantee} holds for the Bregman 
divergence induced by $\rhat$ when we run Algorithm~\ref{alg:innerloop} 
with the gradient estimator~\eqref{eq:tgdef-l2-dynamic}, the distance 
generating function $\rhat$ (i.e., different step sizes for each block) and 
the parameters of Proposition~\ref{prop:innerloop-cbb} with only one  
change: instead of $\tau=\frac{1}{\eta}$ we take $\tau = 
\frac{1}{\eta\rho }$. This keeps the local norms argument valid, as we 
also change the step size in the $\yset$ block from $\eta$ to $\eta 
\rho $; we give a detailed explanation in 
Appendix~\ref{app:rescaled-l2-l1}. 

Finally---as we explain in the previous sections---for the 
guarantee~\eqref{eq:innerloop-guarantee} to imply an 
$(\alpha,0)$-relaxed proximal oracle we require the following property 
from $\phi$ 
and $\psi$:
\begin{equation}\label{eq:phi-psi-regret-convexity-req}
\inner{\phi(z\x)+\psi(z\y)}{z-u}~\text{is convex in $z$ for all $u\in\zset$}.
\end{equation} 
We remark that quadratic and negative entropy functions (and their induced divergences) satisfy this 
requirement, and those are perhaps the most common strongly convex 
composite terms for the geometries of interest. Moreover, we can relax this 
requirement by using restarting  (Algorithm~\ref{alg:highprecision}) for  the 
inner loop of Algorithm~\ref{alg:full-loop}, at the cost of an additional 
logarithmic term in the runtime.

With the algorithmic and analytic adjustments in place, we have the 
following result.

\begin{proposition}
\label{prop:minEB-strong}
For the problem~\eqref{eq:composite-bilinear} 
satisfying~\eqref{eq:composite-bounded},~\eqref{eq:composite-stronly} 
and~\eqref{eq:phi-psi-regret-convexity-req},  
Algorithm~\ref{alg:full-loop} together with the centered gradient 
estimators form Section~\ref{sec:apps} (with $\tau=(\eta 
\sqrt{\mu\x/\mu\y})^{-1}$ in the $\ell_2$-$\ell_1$ setup) returns a point 
$z$ such that $\E 
\gap(z) \le \epsilon$ and runs in 
time
\begin{equation*}
O\left( \nnz(A) + \sqrt{\nnz(A)\cdot (m+n)} 
\cdot \frac{L}{\sqrt{\mu\x\mu\y}}\cdot 
\log\frac{(B+L)\sqrt{\hat{\Theta}}}{\epsilon}
\right),
\end{equation*}
where $\hat{\Theta}= \left(\sqrt{\mu\x/\mu\y} + 
\sqrt{\mu\y/\mu\x}\right)^2\cdot \Theta$, 
\begin{equation*}
L = 
\begin{cases}
\norm{A}_{\max} & \mbox{in the $\ell_1$-$\ell_1$ setup} \\
\max_{i\in[n]}\ltwo{\ai} & \mbox{in the $\ell_2$-$\ell_1$ setup} \\
\lfro{A} & \mbox{in the $\ell_2$-$\ell_2$ setup}  \\
\end{cases}
~~\mbox{and}~~
\Theta = 
\begin{cases}
\log(mn) & \mbox{in the $\ell_1$-$\ell_1$ setup} \\
\log(2m) & \mbox{in the $\ell_2$-$\ell_1$ setup} \\
1 & \mbox{in the $\ell_2$-$\ell_2$ setup}.  \\
\end{cases}
\end{equation*}
\end{proposition}

\begin{proof}
For each setup, Proposition~\ref{prop:innerloopproof} and 
Proposition~\ref{prop:innerloop-cbb} (for the $\ell_2$-$\ell_1$ setup) 
together with the assumption~\eqref{eq:phi-psi-regret-convexity-req} 
guarantee that the $L$-centered gradient estimators of 
Section~\ref{sec:apps} implement an $(\alpha,0)$-relaxed proximal oracle 
with respect to the rescaled distance generating function~\eqref{eq:rescaled-r}, when
we choose $\eta$ appropriately and take $T=O(L^2/\alpha^2)$. 
Also, we have $\norm{(\grad_x f(z), -\grad_y f(z'))}_* = 
\norm{(A^\top z\y, -A {z'}\x)}_* \le O(L)$ for all $z,z'\in\zset$ in every 
setup. Therefore, using the boundedness 
assumption~\eqref{eq:composite-bounded} and 
Proposition~\ref{prop:outerloopproof-sm} with $\rho=\sqrt{\mu\x/\mu\y}$, $\mu=\sqrt{\mu\x\mu\y}$, and $G=L+B$ gives
\begin{equation*}
\begin{aligned}
	 \E \gap(z_K) \le G\, \E \norm{z_K - z^\star} &  \le \sqrt{2}G \sqrt{ \left(\rho+\frac{1}{\rho}\right)\E\left[ \rho V\x_{z_K}(z^\star)+\frac{1}{\rho} V\y_{z_K}(z^\star)\right]}\\
	 & \le \sqrt{2}G \sqrt{ \left(\rho+\frac{1}{\rho}\right)^2\frac{\alpha^K\Theta}{(\mu+\alpha)^K}}.
\end{aligned}
\end{equation*} 
By choosing
$K=O\Big(\frac{\alpha}{\mu}\log \big[
\frac{(B+L){\sqrt{\Theta}}}{\epsilon}(\rho+\rho^{-1})\big]\Big)$, we have $\E\gap(z_K)\le\eps$. Moreover, since we assume $\phi$ and 
$\psi$ admit efficient proximal mappings, each inner loop iteration of 
Algorithm~\ref{alg:full-loop} runs in time $O(n+m)$ and consequently the 
whole inner loop  
runs in time $O( \frac{L^2}{\alpha^2}(m+n))$. Each outer loop requires 
additional $O(\nnz(A))$ time for the exact gradient computations. The 
overall runtime is therefore
\begin{equation*}
O\left( \left[\nnz(A) + (m+n)\frac{L^2}{\alpha^2}\right]
\cdot \frac{\alpha}{\mu} \cdot \log\frac{(B+L)\sqrt{\hat{\Theta}}}{\epsilon}\right).
\end{equation*}
Choosing $\alpha = L \sqrt{\frac{m+n}{\nnz(A)}}$ establishes the claimed 
runtime bound.
\end{proof}

To conclude this section we remark that that adding an additional 
composite term of the form $\phi=\frac{\veps}{\Theta} r\x$ changes the 
objective by at most $\veps$ everywhere in the domain. Therefore, if the 
condition~\eqref{eq:composite-stronly} holds with $\mu\x=0$ and 
$\mu\y > 0$, we can solve the problem to $\epsilon$ accuracy in time 
$\Otil{\nnz(A) + \sqrt{\nnz(A)\cdot 
(m+n)}\frac{L\sqrt{\Theta}}{\sqrt{\mu\y \epsilon}}}$ 
by adding the above regularizer with $\veps =\epsilon/2$ and running 
Algorithm~\ref{alg:full-loop}.

\subsection{Finite-sum problems}
\label{ssec:finitesum}

Consider finite-sum saddle point problems of the form
\[
\min_{x\in\xset}\max_{y\in\yset} f(x,y),~\mbox{where}~
f(x,y)= \frac{1}{K}\sum_{k\in[K]}f_k(x,y)
\]
is convex in $x$ and concave in $y$, and each component $f_k$ 
has $L_k$-Lipschitz gradient mapping
\begin{equation*}
g_k(x,y)\defeq(\grad_x f_k(x,y),-\grad_y f_k(x,y)).
\end{equation*}
This is the standard setting in the literature on variance reduction, where 
the following stochastic gradient estimator is familiar~\cite{Johnson13}. 
For reference point $w_0\in\zset$, draw index $k$ with probability
\begin{equation*}
p_k \defeq \frac{L_k}{\sum_{k’\in[K]}L_{k'}}
\end{equation*}
and set
\begin{equation}
\label{eq:tgdef-finitesum}
	\tilde{g}_{w_0}(w)=\frac{1}{p_k\cdot 
	K}\left(g_k(w)-g_k(w_0)\right)+\frac{1}{K}\sum_{k\in[K]}g_k(w_0).
\end{equation}
This gradient estimator is centered according to 
Definition~\ref{def:w-l-centered}.
\begin{lemma}
	The estimator~\eqref{eq:tgdef-finitesum} is $(w_0,L)$-centered with 
	$L=\frac{1}{K}\sum_{k\in[K]}L_k$.
\end{lemma}
\begin{proof}
The first property follows immediately. The second property follows by 
noting that
\begin{flalign*}
	\E\norm{\tilde{g}_{w_0}(w)-g(w_0)}_*^2 
	& =\sum_{k\in[K]}\frac{1}{p_k K^2} 		   
	   \norm{g_k(w)-g_k(w_0)}_*^2 \\&
	\stackrel{(i)}{\le} 
	\sum_{k\in[K]}\frac{1}{p_k}\cdot\frac{L_k^2}{K^2}\norm{w-w_0}^2 
	\stackrel{(ii)}{\le} 
	\Bigg(\frac{1}{K}\sum_{k\in[K]}L_k\Bigg)^2\norm{w-w_0}^2,
\end{flalign*}
where we use $(i)$ that $g_k$ is $L_k$-Lipschitz and $(ii)$ the choice of 
$p_k$. 
\end{proof}

\newcommand{\Tgrad}{\mathsf{T}_\grad}

Using Algorithm~\ref{alg:highprecision} and the gradient 
estimator~\eqref{eq:tgdef-finitesum} we can implement a relaxed proximal 
oracle (if $x\mapsto\inners{\sum_{k\in[K]} g_k(x)}{x-u}$ is convex for
every 
$u$, as in the bilinear case, we can use Algorithm~\ref{alg:innerloop} 
instead, improving the runtime by a logarithmic factor). Combining this oracle with Algorithm~\ref{alg:outerloop} and 
using $\alpha=L/\sqrt{K}$, we 
solve the finite-sum problem to $\epsilon$ accuracy (i.e., find a point $z$ 
such that $\E \gap(z) \le \epsilon$) in time
\begin{equation*}
\Otilb{ \left[K + \sqrt{K}\frac{L\Theta}{\eps}\right]\Tgrad}
\end{equation*}  
where $\Tgrad$ is the maximum of $m,n$ and the time to compute $g_k$ 
for any single $k\in[K]$. In contrast, solving the problem using the 
mirror-prox method \cite{Nemirovski04} takes time $\O{ \frac{L\Theta}{\epsilon} K\Tgrad}$ 
which can be significantly higher when $K$ is large. 

If in addition $\frac{1}{K}\sum_{k\in[K]}g_k$ is $\mu$-strongly monotone 
relative to the distance generating function (see 
Section~\ref{ssec:strongly}), using 
Algorithm~\ref{alg:outerloop-sm} instead of Algorithm~\ref{alg:outerloop} 
allows us to find an $\epsilon$-accurate solution in time
\begin{equation*}
\Otilb{ \left[K + \sqrt{K}\frac{L}{\mu}\right]\Tgrad}.
\end{equation*}   %
\subsection{Additional gradient estimators}\label{sec:additional}

We revisit the bilinear setting studied in 
Section~\ref{sec:apps} and provide additional gradient estimators that meet 
our variance requirements. 
 In 
Section~\ref{ssec:l1l2-add} we consider $\ell_2$-$\ell_1$ games 
and construct an ``oblivious'' estimator for the $\yset$ component of the 
gradient that involves sampling from a distribution independent of the 
query point. In  
Section~\ref{ssec:l2l2-add} we describe two additional 
centered gradient estimators for $\ell_2$-$\ell_2$ games; one of them is 
the ``factored splits'' estimator proposed in~\citep{BalamuruganB16}.

\subsubsection{$\ell_2$-$\ell_1$ games}\label{ssec:l1l2-add}
Consider the $\ell_2$-$\ell_1$ setup described in the beginning of 
Section~\ref{ssec:l1l2}. 
We describe an alternative for the $\yset$ component 
of~\eqref{eq:tgdef-l2-dynamic}, that is ``oblivious'' in the sense that it 
involves sampling from distributions that do not depend on the current 
iterate. The estimator generates each coordinate of $\tilde{g}_{w_0}\y$ 
independently in the following way: for every $i\in[m]$ we define the 
probability 
$q^{(i)}\in\Delta^n$ by
\begin{equation*}
q^{(i)}_j = {A_{ij}^2}/{\ltwo{\ai}^2},~~\forall j\in[n].
\end{equation*}
Then, independently for every $i\in[m]$, draw $j(i)\sim q^{(i)}$ and set 
\begin{equation}
\label{eq:tgdef-l2-oblivious}
[\tilde{g}_{w_0}\y\left(w\right)]_i =
-[Aw\x_0]_i -\trunc\left( 
A_{ij(i)}\frac{[w\x]_{j(i)}-[w_0\x]_{j(i)}}{q_{j(i)}^{(i)}}\right)
,
\end{equation}
where $\trunc$ is the clipping operator defined 
in~\eqref{eq:tgdef-l2-dynamic}. Note that despite requiring $m$ 
independent samples from different distributions over $n$ elements, 
$\tilde{g}_{w_0}\y$ still admits efficient evaluation. This is because the 
distributions $q\sind{i}$ are fixed in advance, and we can pre-process them 
to perform each of the $m$ samples in time  $O(1)$~\citep{Vose91}. 
However, the oblivious gradient estimator produces fully dense estimates 
regardless of the sparsity of $A$, which may prohibit further runtime improvements when the maximum number of nonzero 
elements in columns of $A$ is smaller than the dimension $m$.

The oblivious 
estimator has the same ``centered-bounded-biased'' properties 
(Definition~\ref{def:z-l-cbb}) as the ``dynamic'' 
estimator~\eqref{eq:tgdef-l2-dynamic}.

\begin{restatable}{lemma}{restateClipped}
	\label{lem:l2-gradient-est-oblivious}
	In the $\ell_2$-$\ell_1$ setup, a gradient estimator with $\xset$ block 
	as in~\eqref{eq:tgdef-l2-dynamic} and $\yset$ block as 
	in~\eqref{eq:tgdef-l2-oblivious} is $(w_0,L,\tau)$-CBB with $L = 
	\nA$.
\end{restatable}

\begin{proof}
	We show the bias bound similarly to the proof of  
	Lemma~\ref{lem:l2-gradient-est-dynamic},
	\begin{align*}
	\left|\Ex{}\left[\tilde{g}_{w_0}\y\left(w\right)-g\y\left(w\right)\right]_{i}\right|
	\le\sum_{j\in\mathcal{J}_{\tau}\left(i\right)}\left|A_{ij}\right|\left|[w\x]_{j}-[w_0\x]_{j}\right|
	\end{align*}
	for all $i\in[m]$, where
	\[
	\mathcal{J}_{\tau}\left(i\right)=\left\{ j\in[n]\mid 
	\trunc\left(\frac{A_{ij}}{q_j\sind{i}}\left([w\x]_{j}-[w_0\x]_{j}\right)\right)\ne
	\frac{A_{ij}}{q_j\sind{i}}\left([w\x]_{j}-[w_0\x]_{j}\right)\right\} .
	\]
	Note that $j\in\mathcal{J}_{\tau}\left(i\right)$ if and only if
	\[
	\left|\frac{A_{ij}}{q_j\sind{i}}\left([w\x]_{j}-[w_0\x]_{j}\right)\right| = 
	\frac{\norm{\ai}_2^2\left|[w\x]_j-[w_0\x]_j\right|}{\left|A_{ij}\right|}>\tau
	\Rightarrow
	\left|A_{ij}\right|\le\frac{1}{\tau}\norm{\ai}_2^2\left|[w\x]_j-[w_0\x]_j\right|.
	\]
	Therefore,
	\[
	\sum_{j\in\mathcal{J}_{\tau}\left(i\right)}\left|A_{ij}\right|\left|[w\x]_j-[w_0\x]_j\right|\le\frac{1}{\tau}\norm{\ai}_2^2\sum_{j\in\mathcal{J}_{\tau}}\left|[w\x]_j-[w_0\x]_j\right|^2=\frac{1}{\tau}\left\Vert
	\ai\right\Vert_2 ^{2}\left\Vert w\x-w_0\x\right\Vert _{2}^{2}
	\]
	and $\linf{\E \tilde{g}_{w_0}\y(w) - g\y(w)} \le 
	\frac{L^2}{\tau}\ltwo{w\x-w_0\x}^2$ follows by taking the maximum 
	over $i\in[m]$. 
	
	The second property follows from the definition of $\trunc$. 
	For the third property, note that the bound~\eqref{eq:l2l1-centered-x} 
	on the $\xset$ component still holds, and  that for each $i\in[m]$ we 
	have 
	$q_j\sind{i}={A_{ij}^2}/{\norm{\ai}_2^2}$ and 
	\begin{align*}
	\Ex{}\left[\tilde{g}_{w_0}\y\left(w\right)-g\y\left(w\right)\right]_i^2
	& 
	=\sum_{j\in[n]}q_j\sind{i}\left(\trunc\left(\frac{A_{ij}}{q_j\sind{i}}\left([w\x]_{j}-[w_0\x]_{j}\right)\right
	)\right)^{2}\\
	& 
	\le\sum_{j\in[n]}q_j\sind{i}\left(\frac{A_{ij}}{q_j\sind{i}}\left([w\x]_{j}-[w_0\x]_{j}\right)\right)^{2}=\left\Vert
	 \ai\right\Vert _2^{2}\left\Vert w\x-w_0\x\right\Vert _{2}^{2}.
	\end{align*}
\end{proof}

 \subsubsection{$\ell_2$-$\ell_2$ games}\label{ssec:l2l2-add}
 In the $\ell_2$-$\ell_2$ setup described in Section~\ref{ssec:l2l2} it is  
 possible to use a completely oblivious gradient estimator. It has the 
 form~\eqref{eq:tgdef} with the following sampling distributions that do 
 not depend on $w_0,w$,
 \begin{equation}
 \label{eq:l2l2-oblivious-probs}
 p_i=\frac{\ltwo{\ai}^2}{\lfro{A}^2}
 ~~\mbox{and}~~
 q_j=\frac{\ltwo{\aj}^2}{\lfro{A}^2}.
 \end{equation}
 \citet{BalamuruganB16} use these sampling distributions, referring to them 
 as ``factored splits.'' The gradient estimator resulting 
 from~\eqref{eq:l2l2-oblivious-probs} is also a special case of the gradient 
 estimator~\eqref{eq:tgdef-finitesum} described above. 
 Another option is to use the dynamic sampling probabilities
  \begin{equation}
 \label{eq:l2l2-var-probs}
 p_i(w)=\frac{\ltwo{\ai} \left| [w\y]_i-[w_0\y]_i\right|}{
 	\sum_{i'\in[m]}\ltwo{\ai[i']}\left| [w\y]_{i'}-[w_0\y]_{i'}\right|}
 ~~\mbox{and}~~
 q_j(w)=\frac{\ltwo{\aj} \left| [w\x]_j-[w_0\x]_j\right|}{
 	\sum_{j'\in[n]}\ltwo{\aj[j']}\left| [w\x]_{j'}-[w_0\x]_{j'}\right|}.
 \end{equation}
 Both the distributions above yield centered gradient estimators.

\begin{lemma}
	In the $\ell_2$-$\ell_2$ setup, the estimator~\eqref{eq:tgdef} with 
	either sampling probabilities~\eqref{eq:l2l2-oblivious-probs} 
	or~\eqref{eq:l2l2-var-probs} is $(w_0,L)$-centered for 
	$L=\lfro{A}$.
\end{lemma}

\begin{proof}
	Unbiasedness follows from the estimator definition. For the oblivious 
	sampling strategy~\eqref{eq:l2l2-oblivious-probs} the second property 
	follows from
	\begin{align*}
	\Ex{}\norm{\tilde{g}_{w_0}(w) - g(w_0)}_{2}^2
	& = \sum\limits_{i \in [m]} 
	\frac{\norm{\ai}_2^2}{p_i}([w\y]_i-[w_0\y]_i)^2 + \sum\limits_{j\in[n]} 
	\frac{\norm{\aj}_2^2}{q_j}([w\x]_j-[w_0\x]_j)^2\\
	& = \norm{A}_\mathrm{F}^2 \norm{w - w_0}_2^2.
	\end{align*}
	For the dynamic sampling strategy~\eqref{eq:l2l2-var-probs}, we have
	\begin{align*}
	\Ex{}\norm{\tilde{g}_{w_0}(w) - g(w_0)}_{2}^2
	& = \Bigg(\sum_{i'\in[m]}\ltwo{\ai[i']}\left| 
	[w\y]_{i'}-[w_0\y]_{i'}\right|\Bigg)^2 + 
	\Bigg(\sum_{j'\in[n]}\ltwo{\aj[j']}\left| 
	[w\x]_{j'}-[w_0\x]_{j'}\right|\Bigg)^2
	\\
	& \le \norm{A}_\mathrm{F}^2 \norm{w - w_0}_2^2,
	\end{align*}
	where the inequality is due to Cauchy–Schwarz.
\end{proof}
We remark that out of the three sampling 
strategies~\eqref{eq:l2l2-probs},~\eqref{eq:l2l2-oblivious-probs} 
and~\eqref{eq:l2l2-var-probs}, only for~\eqref{eq:l2l2-var-probs} the 
bound $\Ex{}\norm{\tilde{g}_{w_0}(w) - g(w_0)}_{2}^2 \le 
\norm{A}_\mathrm{F}^2 \norm{w - w_0}_2^2$ is an inequality, whereas for 
the other two it holds with equality. Consequently, the dynamic sampling 
probabilities~(\ref{eq:l2l2-var-probs}) might be preferable in certain cases.

\newpage

\arxiv{\subsection*{Acknowledgment}}
\notarxiv{\subsubsection*{Acknowledgments}}
YC and YJ were supported by Stanford Graduate Fellowships. AS was 
supported by the NSF CAREER Award CCF-1844855. KT was supported by 
the NSF Graduate Fellowship DGE1656518.

\bibliographystyle{abbrvnat} %

\newpage

\appendix
\part*{Appendix}
\section{Standard results}
Below we give two standard results in convex optimization: bounding 
suboptimality via regret (Section~\ref{app:gap-bound}) and the mirror 
descent regret bound (Section~\ref{app:mirror-descent}).

\subsection{Duality gap bound}
\label{app:gap-bound}

Let $f:\xset\times\yset\to \R $ be convex in $\xset$, concave in $\yset$ 
and differentiable, and let $g(z)=g(x,y)=(\grad_x f(x,y), -\grad_y f(x,y))$. 
For $z,u\in \zset$ define
\begin{equation*}
\gapu(z;u) \defeq f(z\x,u\y)-f(u\x,z\y) 
~~\mbox{and}~~
\gap(z) \defeq \max_{u\in\zset}\gapu(z;u).
\end{equation*}

\begin{lemma}\label{lem:gap-bound}
	For every $z_1,\ldots,z_K\in\zset$,
	\begin{equation*}
	\gap\left(\frac{1}{K}\sum_{k=1}^K z_k\right)\le \max_{u\in \zset} 
	\frac{1}{K}\sum_{k=1}^K \inner{g(z_k)}{z_k-u}.
	\end{equation*}
\end{lemma}
\begin{proof}
Note that $\gapu(z;u)$ is concave in $u$ for every $z$, and that 
$\gapu(z;z)=0$, therefore
\begin{equation*}
\gapu(z;u) \le \inner{\grad_u \gapu(z;z)}{u-z} = \inner{g(z)}{z-u}.
\end{equation*} 
Moreover, $\gapu(z;u)$ is convex in $z$ for every $u$. 
Therefore, for a sequence $z_1, \ldots, z_K$ and any $u\in\zset$
\begin{equation*}
\gapu\bigg(\frac{1}{K}\sum_{k=1}^K z_k;u\bigg) \le 
\frac{1}{K}\sum_{k=1}^K\gapu(z_k;u) \le 
\frac{1}{K}\sum_{k=1}^K\inner{g(z_k)}{z_k-u}.
\end{equation*}
Maximizing the inequality over $u$ yields the lemma.
\end{proof}

\subsection{The mirror descent regret bound}\label{app:mirror-descent}
Recall that $V_{z}(z')=r(z')-r(z)-\inner{\grad r(z)}{z'-z}$ is the Bregman 
divergence induced by a 1-strongly-convex distance generating function 
$r$.

\begin{lemma}\label{lem:mirror-descent}
	Let $Q:\zset\to\R$ be convex, let $T\in\N$ and let 
	$w_0\in\zset$, $\gamma_0,\gamma_1,\ldots,\gamma_T\in\zset^*$. 
	The sequence 
	$w_1,\ldots,w_T$ defined 
	by
	\begin{equation*}
	w_{t} = \argmin_{w\in\zset}\left\{ 
	\inner{\gamma_{t-1}}{w}+Q(w)+V_{w_{t-1}}(w)
		\right\}
	\end{equation*}
	satisfies for all $u\in\zset$ (denoting $w_{T+1}\defeq u$),
	\begin{flalign*}
	\sum_{t=1}^T\inner{\gamma_t + \grad Q(w_t)}{w_t -u}
	&\le 
	V_{w_0}(u) +
	\sum_{t=0}^{T}
	\left\{  \inner{\gamma_{t}}{w_{t} - w_{t+1}}  - V_{w_{t}}(w_{t+1})\right\}
	\\ & \le
	V_{w_0}(u) + \frac{1}{2}\sum_{t=0}^T \norm{\gamma_t}_*^2.
	\end{flalign*}
\end{lemma}

\begin{proof}
	Fix $u\equiv w_{T+1}\in\zset$. We note that by definition $w_t$ is the 
	solution 
	of a 
	convex optimization problem with (sub)gradient $\gamma_{t-1} + \grad 
	Q(\cdot) + \grad V_{w_{t-1}}(\cdot)$, 
	and therefore by by the first-order optimality condition \cite[cf.][Chapter 
	VII]{HiriartUrrutyL93} satisfies
	\begin{equation*}
	\inner{\gamma_{t-1} + \grad Q(w_t) + \grad 
	V_{w_{t-1}}(w_t)}{w_t-w_{T+1}} \le 
	0.
	\end{equation*}
	By the equality~\eqref{eq:three-point} we have $-\inner{\grad 
	V_{w_{t-1}}(w_t)}{w_t-w_{T+1}} = 
	V_{w_{t-1}}(w_{T+1}) - 
	V_{w_{t}}(w_{T+1}) - V_{w_{t-1}}(w_t)$. Substituting and summing over 
	$t\in[T]$ gives
	\begin{equation*}
	\sum_{t=1}^{T}\inner{\gamma_{t-1} + \grad Q(w_t)}{w_t-w_{T+1}} \le 
	V_{w_0}(w_{T+1}) - \sum_{t=0}^{T} V_{w_{t}}(w_{t+1}).
	\end{equation*}
	Rearranging the LHS and adding $\inner{\gamma_T}{w_T - w_{T+1}}$ to 
	both 
	sides of the inequality gives
	\begin{equation*}
	\sum_{t=1}^{T}\inner{\gamma_{t} + \grad Q(w_t)}{w_t-w_{T+1}} \le 
	V_{w_0}(w_{T+1}) +
	\sum_{t=0}^{T}
	\left\{  \inner{\gamma_{t}}{w_{t} - w_{t+1}}  - V_{w_{t}}(w_{t+1})\right\},
	\end{equation*}
	which is the first bound stated in the lemma. The second bound follows 
	since for every $t$ we have
	\begin{equation}\label{eq:holder-young-strong}
	\inner{\gamma_{t}}{w_{t} - w_{t+1}} 
	\stackrel{(i)}{\le} 
	\norm{\gamma_t}_* \norm{w_{t} - w_{t+1}}
	\stackrel{(ii)}{\le} 
	\half \norm{\gamma_t}_*^2 + \half\norm{w_{t} - w_{t+1}}^2
	\stackrel{(iii)}{\le} 
	\half \norm{\gamma_t}_*^2 +V_{w_{t}}(w_{t+1})
	\end{equation}
	due to $(i)$ Hölder's inquality, $(ii)$ Young's inequality and $(iii)$ strong 
	convexity of $r$.
\end{proof} %
\notarxiv{
	\section{Proofs from~\Cref{sec:framework}}
	
	\subsection{Derivation of the Nemirovski's conceptual 
	prox-method}\label{app:outerloop}
	\restateOuterLoop*
	\outloopproof
	
	\subsection{Proof of 
		Proposition~\ref{prop:innerloopproof}}\label{app:innerloop}
}

\arxiv{
\section{Proof of 
Proposition~\ref{prop:innerloopproof}}\label{app:innerloop}
}

\restateInnerLoop*

\begin{proof}
	Recall the expression $w_t = 
	\argmin_{w\in\zset}\left\{\inner{\eta\tilde{g}_{w_0}(w_{t - 
			1})}{w} + \frac{\eta\alpha}{2}V_{w_0}(w) + V_{w_{t - 1}}(w) 
	\right\}$ for the iterates of Algorithm~\ref{alg:innerloop}. We apply 
	Lemma~\ref{lem:mirror-descent} with $Q(z) = \inner{g(w_0)}{z} + 
	\tfrac{\alpha}{2}V_{w_0}(z)$ and $\gamma_t=\eta \tilde{\delta}_t$,
	where
	\begin{equation*}
	\tilde{\delta}_t = \tilde{g}_{w_0}(w_t) - g(w_0).
	\end{equation*}
	Dividing through by $\eta$, the resulting regret bound reads
	\begin{equation}\label{eq:composite-regret-basic-rep}
	\sum_{t\in[T]} \inner{\tilde{g}_{w_0}(w_t) + \tfrac{\alpha}{2}\grad 
		V_{w_0}(w_t)}{w_t-u} \le \frac{V_{w_0}(u)}{\eta} 
	+ \frac{\eta}{2}\sum_{t\in[T]} \norms{\tilde{\delta}_t}_*^2,
	\end{equation}
	where we used the fact that $\tilde{\delta}_0 = 0$ to drop the 
	summation over $t=0$ in the RHS. Now, let
	\begin{equation*}
	\tilde{\Delta}_t = g(w_t) - \tilde{g}_{w_0}(w_t).
	\end{equation*}
	Rearranging the inequality~\eqref{eq:composite-regret-basic-rep}, we 
	may write it as 
	\begin{equation}\label{eq:composite-regret-ghost}
	\sum_{t\in[T]} \inner{g(w_t) + \tfrac{\alpha}{2}\grad 
		V_{w_0}(w_t)}{w_t-u} 
	\le
	 \frac{V_{w_0}(u)}{\eta} 
	+ \frac{\eta}{2}\sum_{t\in[T]} \norms{\tilde{\delta}_t}_*^2
	+ \sum_{t\in[T]}  \inners{\tilde{\Delta}_t}{w_t-u}.
	\end{equation}
	Define the ``ghost iterate'' sequence $s_1, s_2, \ldots, s_T$ according to
	\begin{equation*}
	s_t = \argmin_{s\in\zset}\left\{\inners{\eta \tilde{\Delta}_{t-1}}{s} + 
	V_{s_{t - 1}}(s) 
	\right\}
	~~\text{with}~~s_0 = w_0.
	\end{equation*}
	Applying Lemma~\ref{lem:mirror-descent} with $Q=0$ and 
	$\gamma_t=\eta \tilde{\Delta}_t$, we have
	\begin{equation}\label{eq:ghost-regret}
	\sum_{t\in[T]}  \inners{\tilde{\Delta}_t}{s_t-u}
	\le 
	\frac{V_{w_0}(u)}{\eta} 
	+ \frac{\eta}{2}\sum_{t\in[T]} \norms{\tilde{\Delta}_t}_*^2,
	\end{equation} 
	where here too we used $\tilde{\Delta}_0=0$. Writing 
	$ \inners{\tilde{\Delta}_t}{w_t-u} =  \inners{\tilde{\Delta}_t}{w_t-s_t}
	+  \inners{\tilde{\Delta}_t}{s_t-u}$ and 
	substituting~\eqref{eq:ghost-regret} 
	into~\eqref{eq:composite-regret-ghost} we have
	\begin{equation*}
	\sum_{t\in[T]} \inner{g(w_t) + \tfrac{\alpha}{2}\grad 
	 	V_{w_0}(w_t)}{w_t-u} 
	\le
	\frac{2V_{w_0}(u)}{\eta} 
	+ \frac{\eta}{2}\sum_{t\in[T]} \left[\norms{\tilde{\delta}_t}_*^2 + 
	\norms{\tilde{\Delta}_t}_*^2\right]
	+ \sum_{t\in[T]}  \inners{\tilde{\Delta}_t}{w_t-s_t}.
	\end{equation*}
	Substituting  
	\begin{equation*}
	-\tfrac{\alpha}{2}\inner{\grad V_{w_0}(w_t)}{w_t-u}
	=
	\tfrac{\alpha}{2}V_{w_0}(u) - \tfrac{\alpha}{2}V_{w_t}(u) - 
	\tfrac{\alpha}{2}V_{w_0}(w_t) 
	\le 
	\tfrac{\alpha}{2}V_{w_0}(u) -  \tfrac{\alpha}{2}V_{w_0}(w_t)
	\end{equation*}
	and dividing by $T$,  we have
	\begin{equation*}
	\frac{1}{T}\sum_{t\in[T]} \inner{g(w_t)}{w_t-u} 
	  \le 
	  \left(\tfrac{2}{\eta T}+\tfrac{\alpha}{2}\right)V_{w_0}(u) + 
	  \frac{1}{T}\sum_{t\in[T]} \left[
	  	\tfrac{\eta}{2}\norms{\tilde{\delta}_t}_*^2 +
	  	\tfrac{\eta}{2}\norms{\tilde{\Delta}_t}_*^2 -
	  	\tfrac{\alpha}{2} V_{w_0}(w_t) + 
	  	\inners{\tilde{\Delta}_t}{w_t-s_t}\right].
	\end{equation*}
	Subtracting $\alpha V_{w_0}(u)$ from both sides and using 
	$\tfrac{2}{\eta T}-\tfrac{\alpha}{2} \le 0$ due to $T\ge 
	\frac{4}{\eta\alpha}$, we obtain
	\begin{equation*}
	\frac{1}{T}\sum_{t\in[T]} \inner{g(w_t)}{w_t-u} - \alpha V_{w_0}(u)
	\le 
	\frac{1}{T}\sum_{t\in[T]} \left[
	\tfrac{\eta}{2}\norms{\tilde{\delta}_t}_*^2 +
	\tfrac{\eta}{2}\norms{\tilde{\Delta}_t}_*^2 -
	\tfrac{\alpha}{2} V_{w_0}(w_t) + 
	\inners{\tilde{\Delta}_t}{w_t-s_t}\right].
	\end{equation*}
	Note that this inequality holds with probability 1 for all $u$. We may 
	therefore maximize over $u$ and then take expectation, obtaining
	\begin{flalign}\label{eq:innerloop-final-exp-bound}
	&\E \max_{u\in\zset} \bigg\{
	\frac{1}{T}\sum_{t\in[T]} \inner{g(w_t)}{w_t-u} 
	- \alpha V_{w_0}(u)\bigg\}
	\nonumber \\ & \qquad \qquad 
	\le 
	\frac{1}{T}\sum_{t\in[T]} \E \left[
	\tfrac{\eta}{2}\norms{\tilde{\delta}_t}_*^2 +
	\tfrac{\eta}{2}\norms{\tilde{\Delta}_t}_*^2 -
	\tfrac{\alpha}{2} V_{w_0}(w_t) + 
	\inners{\tilde{\Delta}_t}{w_t-s_t}\right].
	\end{flalign}
	It remains to argue the the RHS is nonpositive. By the first centered 
	estimator property, we have
	\begin{equation*}
	\E \big[ \tilde{\Delta}_t \mid w_t, s_t\big] = 
	\E \big[ g(w_t) - 
	\tilde{g}_{w_0}(w_t) \mid w_t, s_t\big]=0
	\end{equation*}
	 and therefore $\E\inners{\tilde{\Delta}_t}{w_t-s_t} =0$ for all $t$. By the 
	 second property 
	 \begin{equation*}
	 \E \norms{\tilde{\delta}_t}_*^2  = 
	 \E \norms{\tilde{g}_{w_0}(w_t) - g(w_0)}_*^2
	 \le 
	 L^2 \norm{w_t - w_0}^2 \le 2L^2 V_{w_0}(w_t),
	 \end{equation*}
	 where the last transition used the strong convexity of $r$. Similarly, by 
	 Lemma~\ref{lem:centered-var} we have
	 \begin{equation*}
	 \E \norms{\tilde{\Delta}_t}_*^2  = 
	 \E \norms{\tilde{g}_{w_0}(w_t) - g(w)}_*^2
	 \le 
	 4L^2 \norm{w_t - w_0}^2 \le 8L^2 V_{w_0}(w_t).
	 \end{equation*}
	 Therefore
	 \begin{equation*}
	 \E\left[ \tfrac{\eta}{2}\norms{\tilde{\delta}_t}_*^2 +
	 \tfrac{\eta}{2}\norms{\tilde{\Delta}_t}_*^2 -
	 \tfrac{\alpha}{2} V_{w_0}(w_t)\right]
	 \le (5\eta L^2 - \tfrac{\alpha}{2})\E V_{w_0}(w_t) = 0,
	 \end{equation*}
	 using $\eta = \frac{\alpha}{10L^2}$.  
\end{proof}

\notarxiv{
\section{The $\ell_1$-$\ell_1$ setup}\label{app:l1l1}

\subsection{Complete pseudo-code}\label{app:l1l1-alg}
\algOneOne

\subsection{Proof of runtime bound}\label{app:l1l1-proof}
\restateThmOneOne*
\proofOneOne
}

\section{The $\ell_2$-$\ell_1$ setup}\label{app:cbb}

\notarxiv{
\subsection{Derivation of gradient clipping}\label{app:l2l1-clip}	
	\presentationTwoOne}

\notarxiv{
With Proposition~\ref{prop:innerloop-cbb} in hand, the proof of 
\Cref{thm:l1l2} follows identically to that of Theorem~\ref{thm:l1l1}, except 
Proposition~\ref{prop:innerloop-cbb} replaces 
Corollary~\ref{cor:innerloop-oracle}, $L$ is now $\nA$ instead of 
$\norm{A}_{\max}$, and  $\Theta = \max_{z'}r(z')- \min_z r(z) = 
\half+\log m \le \log(2m)$ rather than $\log(mn)$.
}

Before giving the proof of  Proposition~\ref{prop:innerloop-cbb} is 
Section~\ref{app:cbb-proof}, we first collect some properties of the KL 
divergence %
in Section~\ref{app:entropy-props}.

\subsection{Local norms bounds}\label{app:entropy-props}

For this subsection, let $\yset$  be the $m$ dimensional simplex 
$\Delta^m$, and let $r(y)=\sum_{i=1}^m y_i \log y_i$ be the negative 
entropy distance generating function. The corresponding Bregman 
divergence is the KL divergence, which is well-defined for any 
$y,y'\in\R^m_{\ge 0}$ and has the form %
\begin{equation}\label{eq:kl-div-formula}
V_{y}(y') = \sum_{i\in[m]} \left[ y_i' \log\frac{y_i'}{y_i}  + y_i - y_i'
\right]= 
\int_{0}^1 dt \int_{0}^{t}   \sum_{i\in[m]} 
\frac{(y_i-y'_i)^2}{(1-\tau)y_i + \tau y'_i}d\tau.
\end{equation}

In the literature, ``local norms'' regret analysis~\citep[Section 
2.8]{Shalev-Shwartz12} 
relies on the fact that
$r^*(\gamma)=\log(\sum_{i} e^{\gamma_i})$ (the conjugate of negative 
entropy in the simplex) is locally smooth with respect 
to a Euclidean 
norm weighted by $\grad r^*(\gamma)  = 
\frac{e^\gamma}{\lone{e^\gamma}}$. More precisely, the Bregman 
divergence $V^*_{\gamma}(\gamma') = r^*(\gamma') - r^*(\gamma) - 
\inner{\grad r^*(\gamma)}{\gamma'-\gamma}$ satisfies
\begin{equation}\label{eq:local-norms}
V^*_{\gamma}(\gamma+\delta) \le \norm{\delta}_{\grad 
r^*(\gamma)}^2 \defeq \sum_{i} [\grad r^*(\gamma)]_i \cdot \delta_i^2~~
\text{whenever }\delta_i\le  1.79~~\forall i.
\end{equation}
Below, we state this bound in a form that is directly applicable to our 
analysis.

\begin{lemma}\label{lem:local-norms-classical}
	Let $y,y'\in\Delta^m$ and $\delta\in\R^m$. If $\delta$ satisfies 
	$\delta_i \le 
	1.79$ for all $i\in[m]$ then the KL divergence $V_{y}(y')$ 
	satisfies 
	\begin{equation*}
	\inner{\delta}{y'-y} - V_{y}(y') \le \norm{\delta}_y^2 \defeq  
	\sum_{i\in[m]} y_i \delta_i^2 
	\end{equation*}
\end{lemma}
\begin{proof}
	It suffices to consider $y$ in the relative interior of the simplex where 
	$r$ is differentiable; the final result will hold for any $y$ in the simplex 
	by continuity.
	Recall the following general facts about convex conjugates:  
	$\inners{\gamma'}{y'}  - r(y') \le r^*(\gamma')$ for any 
	$\gamma'\in\R^m$, $y=\grad r^* (\grad r(y))$ and  
	$r^*(\grad r(y)) = \inner{\grad r(y)}{y}-r(y)$. Therefore, we have for all 
	$y'\in\Delta^m$,
	\begin{flalign*}
	\inner{\delta}{y'-y} - V_{y}(y')  & = \inner{\grad r(y)+\delta}{y'} - r(y')
	- \left[\inner{\grad r(y)}{y} - r(y)\right] - \inner{y}{\delta}
	\\ & \le
	r^*(\grad r(y)+\delta) - r^*(\grad r(y)) - \inner{\grad r^* (\grad 
	r(y))}{\delta} = V_{\grad r(y)}^*\big(\grad r(y)+\delta\big).
	\end{flalign*}
	The result follows from~\eqref{eq:local-norms} with $\gamma=\grad 
	r(y)$, recalling again that $y=\grad r^* (\grad r(y))$. For completeness we 
	prove~\eqref{eq:local-norms} below,
	following~\cite{Shalev-Shwartz12}. We have
	\begin{flalign*}
	r^*(\gamma + \delta) - r^*(\gamma)  & = 
	\log\left(\frac{\sum_{i\in[m]}e^{\gamma_i + 
	\delta_i}}{\sum_{i\in[m]}e^{\gamma_i}}\right)
	\stackrel{(i)}{\le} 
	\log\left(1 + \frac{\sum_{i\in[m]}e^{\gamma_i}(\delta_i + 
	\delta_i^2)}{\sum_{i\in[m]}e^{\gamma_i}}\right)
	\\ & = 
	\log(1+\inner{\grad r^*(\gamma)}{\delta + \delta^2})
	\stackrel{(ii)}{\le} \inner{\grad r^*(\gamma)}{\delta} + \inner{\grad 
	r^*(\gamma)}{\delta^2},
	\end{flalign*}
	where $(i)$ follows from $e^{x} \le 1+x+x^2$ for all $x\le 1.79$ and 
	$(ii)$ follows from $\log(1+x)\le x$ for all $x$. Therefore, 
	\begin{equation*}
	V^*_{\gamma}(\gamma+\delta) = r^*(\gamma + \delta) - r^*(\gamma) 
	-  
	\inner{\grad r^*(\gamma)}{\delta} \le \inner{\grad 
		r^*(\gamma)}{\delta^2} = \norm{\delta}_{\grad r^*(\gamma)}^2,
	\end{equation*}
	completing the proof.
\end{proof}

\subsection{Proof of 
Proposition~\ref{prop:innerloop-cbb}}\label{app:cbb-proof}
\restateInnerLoopCBB*

\begin{proof}
	Let $w_1, ..., w_T$ denote the iterates of Algorithm~\ref{alg:innerloop} 
	and let $w_{T+1}\equiv u$. We recall the following notation from the 
	proof of Proposition~\ref{prop:innerloopproof}: $\tilde{\delta}_t = 
	\tilde{g}_{w_0}(w_t) - g(w_0)$, $\tilde{\Delta}_t = g(w_t) - 
	\tilde{g}_{w_0}(w_t)$ and $s_t = \argmin_{s\in\zset}\left\{\inners{\eta 
	\tilde{\Delta}_{t-1}}{s} + V_{s_{t - 1}}(s) \right\}$.
	Retracing the steps of the proof of Proposition~\ref{prop:innerloopproof} 
	leading up to the bound~\eqref{eq:innerloop-final-exp-bound}, we 
	observe that by using the first inequality in  
	Lemma~\ref{lem:mirror-descent} 
	rather than the second, the bound~\eqref{eq:innerloop-final-exp-bound} 
	becomes
	\begin{flalign}\label{eq:innerloop-cbb-final-exp-bound}
	\E \max_{u\in\zset} \bigg\{
	\frac{1}{T}\sum_{t\in[T]} \inner{g(w_t)}{w_t-u} 
	- \alpha V_{w_0}(u)\bigg\}
	\le \frac{1}{T}\sum_{t\in[T]} \E \left[- \tfrac{\alpha}{2} V_{w_0}(w_t) + 
	\inners{\tilde{\Delta}_t}{w_t-s_t}\right] 
	\nonumber &\\  %
	\qquad + 
	\frac{1}{\eta T}\sum_{t\in[T]} \E \left[
	\inner{\eta \tilde{\delta}_t}{w_{t}-w_{t+1}} - V_{w_{t}}(w_{t+1})
	+
	\inner{\eta \tilde{\Delta}_t}{s_{t}-s_{t+1}} - V_{s_{t}}(s_{t+1})
	\right].&
	\end{flalign}
	
	Let us bound the various expectations in the 
	RHS of~\eqref{eq:innerloop-cbb-final-exp-bound} one by one. By the 
	first CBB 
	property, $E\big[ 
	\tilde{\Delta}_t\x \mid w_t, s_t\big] =0$ and 
	also 
	$\big\| E\big[ 
	\tilde{\Delta}_t\y \mid w_t, s_t\big]\big\|_* \le 
	\frac{L^2}{\tau}\norm{w_t-w_0}^2$. Consequently,
	\begin{equation*}
	\E \inners{\tilde{\Delta}_t}{w_t-s_t} \le 
	\frac{L^2}{\tau}\E \norm{w_t-w_0}^2 \lone{w_t\y-s_t\y}.
	\end{equation*}
	Using $\lone{y-y'}\le 2$ for every $y,y'\in\yset=\Delta^m$ as well as 
	$\tau = \frac{1}{\eta}$, we obtain
	\begin{equation}\label{eq:cbb-bias-bound}
	\E \inners{\tilde{\Delta}_t}{w_t-s_t} \le 
	2\eta L^2 \E \norm{w_t-w_0}^2 
	\le 
	4\eta L^2 \E V_{w_0}(w_t).
	\end{equation}
	
	To bound the expectation of $\inner{\eta \tilde{\delta}_t}{w_{t}-w_{t+1}} 
	- V_{w_{t}}(w_{t+1})$, we write $w_t = (x_t, y_t)$, and note 
	that for the $\ell_2$-$\ell_1$ setup the Bregman divergence is 
	separable, i.e.\ $V_{w_t}(w_{t+1}) = V_{x_t}(x_{t+1}) + V_{y_t}(y_{t+1})$. 
	For the $\xset$ component, we proceed as in 
	Lemma~\ref{lem:mirror-descent}, and write
	\begin{equation*}
	\inners{\eta \tilde{\delta}_t\x}{x_{t}-x_{t+1}} - V_{x_{t}}(x_{t+1})
	\le 
	\tfrac{\eta^2}{2}\ltwos{\tilde{\delta}_t\x}^2.
	\end{equation*}
	For the $\yset$ component, we observe that
	\begin{equation*}
	\linfs{\eta \tilde{\delta}_t\y} = \eta 
	\linfs{\tilde{g}\y_{w_0}(w_t)-g\y(w_0)} \le 
	\eta \tau = 1
	\end{equation*}
	by the second CBB property and $\tau = \frac{1}{\eta}$. 
	Therefore, we may 
	apply Lemma~\ref{lem:local-norms-classical} with $\delta = -\eta 
	\tilde{\delta}_t\y$ and obtain
	\begin{equation*}
	\inners{\eta \tilde{\delta}_t\y}{y_{t}-y_{t+1}} - V_{y_{t}}(y_{t+1})
	\le \eta^2 \sum_{i\in [m]} [y_t]_i [\tilde{\delta}_t\y]_i^2.
	\end{equation*}
	Taking expectation and using the fact that $y_t$ is in the simplex gives
	\begin{equation*}
	\E \left[ \inners{\eta \tilde{\delta}_t\y}{y_{t}-y_{t+1}} - V_{y_{t}}(y_{t+1}) 
	\right] \le 
	\eta^2 \E \max_{i\in[m]} \E \left[[\tilde{\delta}_t\y]_i^2 \mid w_t \right].
	\end{equation*}
	The third CBB property reads $
	\E \left[\ltwos{\tilde{\delta}_t\x}^2 \mid w_t \right] + 
	\max_{i\in[m]} \E \left[[\tilde{\delta}_t\y]_i^2 \mid w_t \right]
	\le L^2 \norm{w_t-w_0}^2$. Therefore, for $t<T$, the above discussion 
	yields
	\begin{flalign}\label{eq:cbb-delta-bound}
	\E\left[\inners{\eta \tilde{\delta}_t}{w_{t}-w_{t+1}} - 
	V_{w_{t}}(w_{t+1})\right] &\le
	\eta^2 \E\left[ \tfrac{1}{2}\ltwos{\tilde{\delta}_t\x}^2 + 
	\max_{i\in[m]} \E \left[[\tilde{\delta}_t\y]_i^2 \mid w_t \right]
	\right]
	\nonumber\\& \le
	 \eta^2 L^2 \E \norm{w_t-w_0}^2
	\le 2\eta ^2 L^2 \E V_{w_0}(w_t).
	\end{flalign}
	
	To bound the expectation of 
	$\inners{\eta\tilde{\Delta}_t}{s_{t}-s_{t+1}} - V_{s_{t}}(s_{t+1})$,
	decompose it as 
	\begin{flalign*}
	\inners{\eta\tilde{\Delta}_t}{s_{t}-s_{t+1}} - V_{s_{t}}(s_{t+1})
	=& ~
	\frac{2}{3}\left[\inners{\tfrac{3}{2}\eta\tilde{\delta}_t}{s_{t}-s_{t+1}} - 
	V_{s_{t}}(s_{t+1})\right]
	\\ &
	+
	\frac{1}{3}\left[\inners{3\eta\left(g(w_t)-g(w_0)\right)}{s_{t}-s_{t+1}} - 
	V_{s_{t}}(s_{t+1})\right]
	\end{flalign*}
	and bound each of the bracketed terms separately.
	For the first term, we note that the bound $\linf{\frac{3}{2}\eta 
	\tilde{\delta}_t} \le \frac{3}{2} \le 1.79$ holds, so we may still apply 
	Lemma~\ref{lem:local-norms-classical} and (substituting 
	$\eta\to\frac{3}{2}\eta$ in the bound~\eqref{eq:cbb-delta-bound}) 
	obtain $\frac{2}{3}\E 
	\left[\inners{\tfrac{3}{2}\eta\tilde{\delta}_t}{s_{t}-s_{t+1}} - 
	V_{s_{t}}(s_{t+1})\right] \le 3\eta^2 L^2$. The bound for the second term 
	follows directly from the Lipschitz continuity of $g$ strong convexity of 
	$r$ (as in the proofs of Lemma~\ref{lem:mirror-descent} and 
	Proposition~\ref{prop:innerloopproof}), giving
	\begin{equation*}
	\frac{1}{3} \E\left[\inners{3\eta\left(g(w_t)-g(w_0)\right)}{s_{t}-s_{t+1}} - 
	V_{s_{t}}(s_{t+1})\right] \le \frac{3}{2}\eta^2 \E \norm{g(w_t)-g(w_0)}_*^2
	\le 3\eta^2 L^2 \E V_{w_0}(w_t).
	\end{equation*}
	Putting the two bounds together, we obtain
	\begin{flalign}\label{eq:cbb-Delta-bound}
	\E\left[\inners{\eta \tilde{\Delta}_t}{s_{t}-s_{t+1}} - 
	V_{s_{t}}(s_{t+1})\right] 
	\le 6\eta^2 L^2 \E V_{w_0}(w_t).
	\end{flalign}

	Substituting~\eqref{eq:cbb-bias-bound},~\eqref{eq:cbb-delta-bound} 
	and~\eqref{eq:cbb-Delta-bound} back 
	into~\eqref{eq:innerloop-cbb-final-exp-bound}, we have
	\begin{equation*}
	\E \max_{u\in\zset} \bigg\{
	\frac{1}{T}\sum_{t\in[T]} \inner{g(w_t)}{w_t-u} 
	- \alpha V_{w_0}(u)\bigg\}
	\le \frac{1}{T}\sum_{t\in[T]} \left[12\eta L^2 - \tfrac{\alpha}{2}\right]
	 \E V_{w_0}(w_t)
	 = 0
	\end{equation*}
	where the last transition follows from $\eta = \frac{\alpha}{24L^2}$; this 
	establishes the bound~\eqref{eq:innerloop-guarantee} for the iterates of 
	Algorithm~\ref{alg:innerloop} with a CBB gradient estimators. By the 
	argument in the proof of Corollary~\ref{cor:innerloop-oracle}, for  
	$g(z)=(A^\top z\y, -Az\x)$,  the average of those iterates constitutes an 
	$(\alpha,0)$-relaxed proximal oracle.
\end{proof}

\subsection{Rescaled Bregman divergence}\label{app:rescaled-l2-l1}
Let $\rho>0$ and, for $z=(x,y)\in\zset$ and $z'=(x',y')\in\zset$, consider 
the rescaled Bregman divergence
\begin{equation}\label{eq:rescaled-div}
\hat{V}_{z}(z') = \rho {V}_{x}\x(x') +  \frac{1}{\rho}{V}_{y}\y(y'), 
\end{equation}
where $V\x_x(x')=\half\ltwo{x-x'}^2$ and ${V}_{y}\y(y')=\sum_{i\in[m]} y'_i 
\log \frac{y_i'}{y_i}$ are the component Bregman divergences of the 
$\ell_2$-$\ell_1$ setup. Consider running Algorithm~\ref{alg:innerloop} 
with the gradient estimator~\eqref{eq:tgdef-l2-dynamic} and the rescaled 
divergence~\eqref{eq:rescaled-div}, using $\eta$ and $T$ as in 
Proposition~\ref{prop:innerloop-cbb} and $\tau=1/(\eta\rho)$. Here 
we prove that this procedure satisfies~\eqref{eq:innerloop-guarantee} with 
$\hat{V}$ instead of $V$; this is useful for the strongly monotone scheme 
we describe in Section~\ref{ssec:strongly}. 

The proof is analogous to the proof of 
Proposition~\ref{prop:innerloop-cbb}, so we only describe the differences 
while using the same notation.
The bound~\eqref{eq:innerloop-cbb-final-exp-bound} is a direct 
consequence of Lemma~\ref{lem:mirror-descent} and therefore holds with 
$\hat{V}$ replacing $V$ since we run Algorithm~\ref{alg:innerloop} with the 
modified divergence (which is equivalent to using different step sizes per 
block). The gradient estimator~\eqref{eq:tgdef-l2-dynamic} satisfies
$\E \inners{\tilde{\Delta}_t}{w_t-s_t} \le 
\frac{L^2}{\tau}\E \ltwo{w_t\x-w_0\x}^2 \lone{w_t\y-s_t\y}.$ Substituting 
$\tau=\frac{1}{\eta\rho}$ and $\rho \ltwo{w_t\x-w_0\x}^2 \le 
2\hat{V}_{w_0}(w_t)$ shows that~\eqref{eq:cbb-bias-bound} holds with 
$\hat{V}$ replacing $V$ as well. Continuing with the proof, we have
\begin{equation*}
\inner{\eta \tilde{\delta}_t}{w_{t}-w_{t+1}} 
- \hat{V}_{w_{t}}(w_{t+1})
=
\Big[\inners{\eta \tilde{\delta}_t\x}{x_{t}-x_{t+1}} - \rho 
V_{x_{t}}(x_{t+1})\Big]
+\Big[\inners{\eta \tilde{\delta}_t\y}{y_{t}-y_{t+1}} - \frac{1}{\rho} 
V_{y_{t}}(y_{t+1})\Big].
\end{equation*}
Treating each bracketed term separately gives
\begin{equation}\label{eq:rescaled-x-bound}
\inners{\eta \tilde{\delta}_t\x}{x_{t}-x_{t+1}} - \rho V_{x_{t}}(x_{t+1})
=
\rho \left( \inners{\rho^{-1}\eta \tilde{\delta}_t\x}{x_{t}-x_{t+1}} -  
V_{x_{t}}(x_{t+1}) \right)
\le 
\tfrac{\eta^2}{2\rho}\ltwos{\tilde{\delta}_t\x}^2
\end{equation}
and, using the local norms bound,
\begin{equation}\label{eq:rescaled-y-bound}
\inners{\eta \tilde{\delta}_t\y}{y_{t}-y_{t+1}} - \frac{1}{\rho}V_{y_{t}}(y_{t+1})
= \frac{1}{\rho} \left( 
\inners{\rho \eta \tilde{\delta}_t\y}{y_{t}-y_{t+1}} - V_{y_{t}}(y_{t+1}) \right)
\le \rho \eta^2 \sum_{i\in [m]} [y_t]_i [\tilde{\delta}_t\y]_i^2.
\end{equation}
Here, we used $\tau=\frac{1}{\eta\rho}$ and $\linfs{\tilde{\delta}_t\y}\le 
\tau$  to guarantee that 
$\linfs{\rho 
\eta \tilde{\delta}_t\y} \le 1.79$ and consequently the above bound is 
valid. Finally, we observe that the gradient 
estimator~\eqref{eq:tgdef-l2-dynamic} satisfies a slightly stronger 
version of the third CBB property, namely
\begin{equation*}
\E \left[\ltwos{\tilde{\delta}_t\x}^2 
\mid w_t \right]  \le L^2\lone{y_t - y_0}^2
~~\mbox{and}~~
\max_{i\in[m]} \E \left[[\tilde{\delta}_t\y]_i^2 \mid w_t \right]
\le L^2 \ltwo{x_t-x_0}^2.
\end{equation*}
Combining this with the bonds~\eqref{eq:rescaled-x-bound} 
and~\eqref{eq:rescaled-y-bound}, we obtain
 \begin{equation*}
 \E \left[ \inner{\eta \tilde{\delta}_t}{w_{t}-w_{t+1}} 
 - \hat{V}_{w_{t}}(w_{t+1}) \right] \le \eta^2 L^2 \left( \tfrac{1}{\rho}\E 
 \lone{y_t - y_0}^2 +  \rho\E \ltwo{x_t - x_0}^2 
 \right) \le 2\eta^2 L^2 \E \hat{V}_{w_0}(w_t),
 \end{equation*}
 so that \eqref{eq:cbb-delta-bound} holds with $\hat{V}$ replacing $V$. 
 The treatment of $\tilde{\Delta}_t$ is completely analogous and thus the 
 claimed result holds.
 
\notarxiv{
\subsection{Complete pseudo-code}\label{app:l2l1-alg}
\algTwoOne
}

\notarxiv{
\section{The $\ell_2$-$\ell_2$ setup}\label{ssec:l2l2}
\presentTwoTwo
}

\end{document}